\documentclass[12pt]{amsart}
\usepackage{graphicx} % Required for inserting images
\usepackage[utf8]{inputenc}
\usepackage{amsmath}
\usepackage{amssymb}
\usepackage{xcolor}
\usepackage{hyperref}
\usepackage{cleveref}
\usepackage{pgfplots}
\usepackage{float}
\pgfplotsset{compat=1.15}
\usepackage{mathrsfs}
\usetikzlibrary{arrows} 
\usepackage[margin = 2.5 cm]{geometry}

\newcommand{\ini}{\operatorname{in}}
\newcommand{\vv}{\mathrm{v}}
\newcommand{\reg}{\mathrm{reg}}

\newtheorem{theorem}{Theorem}[section]
\newtheorem{lemma}[theorem]{Lemma}
\newtheorem{conjecture}[theorem]{Conjecture}
\newtheorem{definition}[theorem]{Definition}
\newtheorem{proposition}[theorem]{Proposition}
\newtheorem{example}[theorem]{Example}
\newtheorem{remark}[theorem]{Remark}
\newtheorem{question}[theorem]{Question}
\newtheorem{corollary}[theorem]{Corollary}
\title{On the $\vv$-number of binomial edge ideals of some classes of graphs}

\author{Deblina Dey}
\author{A. V. Jayanthan}
\address{Department of Mathematics, I.I.T. Madras, Chennai, Tamil Nadu,  India - 600036.}
\email{deblina.math@gmail.com}
\email{jayanav@iitm.ac.in}

\author{Kamalesh Saha}
\address{Chennai Mathematical Institute, Siruseri, Tamil Nadu,  India - 603103.}
\email{ksaha@cmi.ac.in; kamalesh.saha44@gmail.com}
\keywords{$\vv$-number, binomial edge ideals, Cohen-Macaulay closed graph, cycle, binary tree.}

\subjclass{Primary 05E40, 13F20, Secondary 05C25, 05C69.}
\begin{document}
\begin{abstract}
Let $G$ be a finite simple graph, and $J_G$ denote the binomial edge ideal of $G$. In this article, we first compute the $\vv$-number of binomial edge ideals corresponding to Cohen-Macaulay closed graphs. As a consequence, we obtain the $\vv$-number for paths. For cycle and binary tree graphs, we obtain a sharp upper bound for $\vv(J_G)$ using the number of vertices of the graph. We characterize all connected graphs $G$ with $\vv(J_G) = 2$. We show that for a given pair $(k,m), k\leq m$, there exists a graph $G$ with an associated monomial edge ideal $I$ having $\vv$-number equal to $k$ and regularity $m$. If $2k \leq m$, then there exists a binomial edge ideal with $\vv$-number $k$ and regularity $m$. %Then we have studied the pair ($\vv$-number, regularity) for both monomial edge ideals and binomial edge ideals. 
Finally, we compute $\vv$-number of powers of binomial edge ideals with linear resolution, thus proving a conjecture on the $\vv$-number of powers of a graded ideal having linear powers, for the class of binomial edge ideals.
\end{abstract}

\maketitle

%\begin{question}
%\begin{enumerate}
%
%    \item What can we say about $\mathrm{v}(J_G)$, when $G$ is a cycle or path.
%    \item What will be the relation between $\mathrm{v}(J_G)$ and $\mathrm{v}(\mathcal{J}_G)$, where $\mathcal{J}_G$ denotes the generalized binomial edge ideal.
%    \item For parity binomial edge ideal can we describe the $\mathrm{v}$ number combinatorially.
%    \item There are certain graph operation like join of graphs, product of graphs. What will be the behaviour of $\mathrm{v}$ number under those changes.
%    \item Relation between $\mathrm{v}$-number and regularity of certain classes of binomial edge ideals.
%    \item classification of graphs $G$ such that $\mathrm{v}(J_G)=2$.
%\end{enumerate}
%\end{question}
%\begin{enumerate}
%    \item Section 1: Introduction
%    \item Section 2: Preliminaries
%    \item Section 3: CM closed graph
%    \item Section 4: v-number = 2, Bounds on v-number
%    \item Section 5: (v, reg) pair
%    \item Section 6: Proof of conjecture
%\end{enumerate}

\section{Introduction}

An invariant called \textit{$\vv$-number} was introduced by Cooper et al. in \cite{cooper2020generalized}. 
We begin by recalling its definition:

Let $\mathcal{R}=K[x_1,\ldots,x_n]=\bigoplus_{d\geq 0}\mathcal{R}_d$ be a standard graded polynomial ring in $n$ variables over a field $K$ and $I\subset \mathcal{R}$ be a proper graded ideal. %We denote the set of associated primes of $I$ by $\mathrm{Ass}(I)$. 
The \textit{$\mathrm{v}$-number} of $I$, denoted by $\mathrm v(I)$, is defined as
\[\mathrm v(I):=\min\{d\geq 0 : \text{ there exists } f\in \mathcal{R}_d\text{ and } \mathfrak{p}\in \mathrm{Ass}(I) ~\mbox{with}~ (I:f)=\mathfrak{p}\}.\]
For each $\mathfrak{p}\in \mathrm{Ass}(I)$, the $\mathrm v$-number of $I$ at $\mathfrak{p}$, called the \textit{local $\mathrm{v}$-number} of $I$ at $\mathfrak{p}$, denoted by $\mathrm v_\mathfrak{p}(I)$ and is defined as $\mathrm v_\mathfrak{p}(I):=\min\{d\geq 0 : \text{ there exists } f\in \mathcal{R}_d ~\mbox{with} ~(I:f)=\mathfrak{p}\}$. In this sense, we have $\vv(I)=\min\{\vv_{\mathfrak{p}}(I): \mathfrak{p}\in\mathrm{Ass}(I)\}$.

The notion of $\vv$-number has its foundation in coding theory. In particular, it was introduced to investigate the asymptotic behaviour of the minimum distance function of projective Reed-Muller-type codes. These codes play a crucial role in error correction and data transmission. However, beyond its coding-theoretic roots, the $\vv$-number has also revealed significant geometric implications, as noted in  \cite{jaramillo2023connected}. To date, several investigations have been conducted on the $\vv$-number, exploring its properties and applications:
%, see [give all citation here] for a partial list. %Some notable research on the $\vv$-number are following:
%\avj{In my opinion, we need not give a list with different topics, we can combine all these and put in one block (as proposed in the previous line).}
\begin{itemize}
    \item $\vv$-number in coding theory \cite{cstpv20} and \cite{pv23};
	\item $\mathrm{v}$-number of monomial ideals in \cite{bm23}, \cite{civan23}, \cite{v-edge}, \cite{ksvgor23}, \cite{sahacover23} and \cite{ssvmon23};
	\item $\mathrm{v}$-number of binomial edge ideals  in \cite{ambhore2023v} and \cite{jaramillo2023connected};
	\item Asymptotic behaviour of the $\vv$-number in \cite{vp24}, \cite{bms24}, \cite{concav23}, \cite{ficsimon23}, \cite{fs23}, \cite{fsasymp24}, \cite{fd24} and \cite{kns24}.
\end{itemize}

In general computation of the $\vv$-number of a homogeneous ideal is a non-trivial task, particularly for non-monomial ideals.
In this paper, we explore the $\vv$-number of binomial edge ideals corresponding to various classes of graphs. Let $G$ be a simple graph on the vertex set $V(G)=[n]=\{1,\ldots,n\}$ with the edge set $E(G)$. Consider the polynomial ring $R=K[x_{1},\ldots,x_{n},y_{1},\ldots,y_{n}]$ over a field $K$. Then the \textit{binomial edge ideal} of $G$, denoted by $J_{G}$, is the graded ideal of $R$,
\[J_{G}=\big<f_{ij} = x_{i}y_{j}-x_{j}y_{i}\mid \{i,j\}\in E(G)\,\, \text{with}\,\, i<j\big>.\]
%Ever since its introduction in [cite Herzog and Ohtani papers], understanding the interplay between the combinatorial and algebraic properties of graphs and associated binomial edge ideals respectively has been a very active area of research in the past one decade, [citations]

%\noindent The study of binomial edge ideals has indeed garnered significant attention in commutative algebra over the past decade. Numerous investigations have been undertaken on these ideals from multiple viewpoints, and readers are directed to consult the survey article \cite{} and \cite{} for additional insights.
\medskip

While the $\vv$-number of monomial edge ideal has been studied in detail, that of binomial edge ideal is not at all well understood in terms of properties of the graphs and other associated algebraic invariants. For binomial edge ideals, the $\vv$-number has not been computed for even the simplest classes like cycles and paths. In \cite{ambhore2023v} and \cite{jaramillo2023connected}, the authors studied the local $\vv$-number of $J_G$ at the prime ideal corresponding to the empty cutset, $\vv_\emptyset(J_G)$. Jaramillo-Velez and Seccia proved that $\vv_{\emptyset}(J_G)$ is equal to the connecting domination number of $G$, \cite{jaramillo2023connected}. It is worth noting that this value can also be interpreted as the minimum completion number of $G$, as described in \cite{ambhore2023v}. %There are certain overlaps in the results proved in \cite{ambhore2023v} and \cite{jaramillo2023connected}. In this article  we will highlight the results that are likely to be more general and applicable across a broader range of scenarios. 
It has been demonstrated in \cite[corollary 3.12]{ambhore2023v} that the $\vv$-number is additive for binomial edge ideals. Ambhore et al., in \cite{ambhore2023v}, proved that if $J_G$ is a Knutson ideal, then $\vv(J_G)\leq \vv(\ini_{<}(J_G))$. They also gave a characterization of graphs $G$ with $\vv(J_G)=1$. Moreover, they conjectured that $\vv(J_G)\leq\reg(\frac{R}{J_G})$ for any simple graph $G$, \cite[Conjecture 5.3]{ambhore2023v}, and proved the conjecture for several classes of graphs such as chordal graphs and whiskered graphs.% Again, in \cite{jaramillo2023connected}, the authors Connected domination in graphs and v-numbers of binomial edge ideals.
\medskip

In \cite{v-edge}, a combinatorial description of $\vv$-number of square-free monomial ideals was given. However, when it comes to binomial edge ideals, devising a combinatorial expression for the $\vv$-number proves to be quite challenging. Though for simple-looking graphs, the prediction of the $\vv$-number formula may be relatively simple, yet the process of formal proof remains challenging. This study aims to determine the precise formula for certain classes of binomial edge ideals and establish  upper bounds for some classes, with an attempt to anticipate the potential $\vv$-number. Additionally, an investigation is carried out on the pair ($\vv$-number, regularity) for edge ideals of graphs, which is then partially extended to include binomial edge ideals. Finally, we solve a conjecture given in \cite{bms24} for the class of binomial edge ideals. The organization of the paper is outlined as follows.
\medskip

Section \ref{sec:preli} recalls the necessary prerequisites to describe the rest part of the paper. In Section \ref{sec:closedcm}, we give a combinatorial description of all the local $\vv$-numbers of a Cohen-Macaulay binomial edge ideal with quadratic Gr\"{o}bner basis (see Theorem \ref{thm_1}). As an application, we get the exact formula of the $\vv$-number of such ideals in a simpler form:
\medskip

\noindent \textbf{Theorem \ref{thm_2}.} \textit{Let $G$ be a graph such that $J_G$ is Cohen-Macaulay and $J_G$ has a quadratic gr\"obner basis. Then $\mathrm{v}(J_G) =\big\lceil\frac{2(t-1)}{3}\big\rceil$, where $t$ is the number of maximal cliques in $G$. 
}
\medskip

\noindent As a consequence of Theorem \ref{thm_2}, we get $\vv(J_{P_n})=\big\lceil\frac{2(n-2)}{3}\big\rceil$, where $P_n$ denotes the path graph on $n$ vertices (see \Cref{cor_path_v}). In Section \ref{sec:v=2}, we first characterize all graphs whose binomial edge ideals have $\vv$-number $2$, \Cref{thm_3}. %Then we show in \Cref{prop_CM_v} that for any graph $G$ with Cohen-Macaulay $J_{G}$, the number of cut vertices of $G$ is equal to $\vv_{\emptyset}(J_G)$. Also, we discuss the behaviour of the $\vv$-number for decomposable graphs (see \Cref{prop_decom} and \Cref{remark_1}). 
We give a sharp upper bound of the $\vv$-number of binomial edge ideals of binary trees (\Cref{cor_binary_v}) and cycles (\Cref{cor_cycle_v}). In this direction, we conjecture the following:
\medskip

\noindent \textbf{Conjecture \ref{conjcycle}. }%\textit{
{\em Let $C_n$ denote the cycle graph on $n$ vertices and $B_n$ denote the binary tree of level $n$. Then
\begin{enumerate}
    \item $\mathrm{v}(J_{C_n}) = n - \lfloor\frac{n}{3} \rfloor=\lceil \frac{2n}{3}\rceil$ for all $n \geq 6$;

    \item $\mathrm{v}(J_{B_n}) = 2^{n-1} + \mathrm{v}(J_{B_{n-3}})$ for all $n \geq 3$.
\end{enumerate}
}

For several classes of graphs, it has been shown that $\vv(I(G))\leq \reg(R/I(G))$, where $I(G)$ denote the monomial edge ideal of $G$. In view of this connection between these two invariants, it is natural to ask:
\medskip

\noindent 
\textbf{Question:} {\em Given a pair of positive integers $(k, m)$ with $k \leq m$, does there exist a connected graph $G$ and an associated edge ideal $I$ in a polynomial ring $R$ such that $\vv(I) = k$ and $\reg(R/I) = m?$} 

\vskip 2mm
\noindent
In Section \ref{sec:pair}, we first answer this question affirmatively for monomial edge ideals, see \Cref{thm_edge_pair}. For binomial edge ideals, we show that if $m \geq 2k$, there exists a graph $G$ with $\vv(J_G) = k$ and $\reg(R/J_G) = m$:

%Let $I(G)\subseteq \mathcal{R}=K[V(G)]$ denotes the edge ideal of a graph $G$.  we prove Theorem , which tells that for any pair of positive integers $(k,m)$ with $k\leq m$, there exists a connected graph $G$ such that $\big(\mathrm{v}(I(G)),\reg(\frac{\mathcal{R}}{I(G)})\big)=\big(k,m\big)$. The same question can be asked for binomial edge ideals and due to \cite[Conjecture 5.3]{ambhore2023v} taking $k\leq m$ make sense. In this direction, we give a partial answer as follows:
\medskip

\noindent \textbf{Corollary \ref{cor_binom_pair}.}\textit{
Given any pair of integers $(k,m)$ such that $k \geq 1$ and $m \geq 2k$, there exists a connected graph $G$ such that $ \mathrm{v}(J_G) = k$ and $\reg (\frac{R}{J_G}) = m$. 
}
\medskip

\noindent 
Ficarra conjectured that if $I$ is a monomial ideal with linear powers, then $\vv(I^k) = \alpha(I)k-1$, where $\alpha(I)$ is the maximum degree of a minimal monomial generator. 
In \cite{bms24}, Biswas et al.  generalized this conjecture and proposed  %\cite[Conjecture 3.5]{bms24} 
that for a graded ideal $I$ with linear powers, $\vv(I^k)$ attains its possible minimum value for all $k\geq 1$. In Section \ref{sec:conj}, we settle this conjecture for the class of binomial edge ideals (see \Cref{thm_vlp}).

\vskip 2mm \noindent
\textbf{Acknowledgements:} Deblina Dey acknowledge the financial support of Prime Minister Research Fellowship for her research. Kamalesh Saha would like to thank the National Board for Higher Mathematics (India) for the financial support through the NBHM Postdoctoral Fellowship. Again, Kamalesh Saha is partially supported by an Infosys Foundation fellowship.

\section{Preliminaries}\label{sec:preli}
We begin by assembling all notations and terminologies that will be followed throughout the paper. Any graph considered in this paper is finite and simple. Given a graph $G$, $V(G)$ and $E(G)$ denote the vertex set and edge set of $G$ respectively. For $u,v\in V(G)$, $u$ is said to be adjacent to $v$ if $\{u,v\}\in E(G)$. A path between two vertices $u$ and $v$ of $G$ is a sequence of vertices, $u, u_1, \ldots, u_k, v$, where any two consecutive vertices are adjacent, and $G$ is connected if there exists a path between any two vertices of $G$. Let $C(G)$ denote the collection of connected components of $G$ and $c(G) = |C(G)|$. Given any set $S\subseteq V(G)$, $G\setminus S$ denotes the induced subgraph of $G$ on the vertex set $V(G) \setminus S$. For any $v \in V(G)$, let $N_G(v) = \{u \in V(G) : \{u,v\} \in E(G)\}$,  $N_G[v] = N_G(v) \cup \{v\}$, and $\deg_G(v) = |N_G(v)|$.
A vertex $v$ is a cut vertex of $G$ if $c(G\setminus\{v\}) > c(G)$ and a set $S\subseteq V(G) $ is said to be a cut set if $c(G\setminus S) > c(G\setminus (S\setminus \{s\}))$ for every $s\in S$. By $\mathcal{C}(G)$, we denote the collection of all cut sets of $G$.\\
%A set $M \subseteq E(G)$ is an induced matching of $G$ if for any $\{u_1,u_2\}, \{v_1,v_2\}\in M$ $u_i \neq v_j$ and there is no edge in $G$ between $u_i$ and $v_j$, where $i,j\in \{1,2\}.$ 
A clique $F$ of $G$ is a complete subgraph of $G$, and it is called maximal if there is no other clique in $G$ properly containing $F$. A vertex $v$ is called a free vertex if the induced subgraph on $N_G[v]$ is a clique. %We denote by $K_n$ a complete graph on $n$ vertices. %The graph $G$ is complete if any pair of vertices are adjacent to each other. By $K_n$, we denote the complete graph on $n$ vertices.

%\begin{definition}
Let $G_1$ and $G_2$ be two subgraphs of $G$ such that $V(G) = V(G_1)\cup V(G_2)$ and  $G_1\cap G_2 \equiv K_m$ for some $m>0$. Then $G$ is called a {clique sum} of $G_1$ and $G_2$ along $K_m$ and we write $G = G_1 \cup_{K_m} G_2$. If $K_m = \{u\}$, where $u \in V(G)$, then we simply write $G= G_1\cup_u G_2$.
%\end{definition}
%Now, let us recall the definitions of some graph classes that we are going to use.
%\begin{definition}
    A simple graph $G$ is a cone graph if there exist $v\in V(G)$ such that for every $u\neq v\in V(G)$, $\{u,v\} \in E(G)$.
%\end{definition}
%In section $3$, we have studied closed Cohen-Macaulay graphs. Closed graphs are a well-known class of graphs in combinatorics. There are lots of works in combinatorial commutative algebra, where people have studied closed graphs with an algebraic point of view \cite{kiani2012binomial}, \cite{herzog2010binomial}, \cite{ene2011cohen}, \cite{jaramillo2023connected} etc. 
%\begin{definition}
Let $G$ be a simple graph with $V(G) = [n]$. Then $G$ is closed if there exist a ordering on $V(G)$ such that for all integers $1\leq i <j<k \leq n$, if $\{i,k\} \in E(G)$, then $\{i,j\} \in E(G)$ and $\{j,k\} \in E(G)$.
%\end{definition}

%\begin{remark}
%The notations $C_n$ and $P_n$ are the standard notations in combinatorics, used for the cycle graph and path graph on $[n]$ vertices, respectively.
%\end{remark}

Let $G$ be a graph on $[n]$ and 
$R = K[x_1,x_2,\ldots,x_n,y_1,y_2,\ldots,y_n] = \oplus_{d\geq 0}R_d$ be the standard graded polynomial ring on $2n$ variables, where $K$ is a field. A polynomial $f\in R$ is homogeneous if $f\in R_d$ for some $d\geq 0$, and then $\deg(f) = d$. Given any $S\subseteq V(G)$, $m_S$ represents the ideal $\langle x_i,y_i: i\in S \rangle$ and for any $j\in V(G)$, $t_j$ is either $x_j$ or $y_j$.
%\begin{remark} \label{remark_2 } 
On the polynomial ring $R$, we take the lexicographic order $<$, induced by the monomial ordering
\[ x_1>x_2> \cdots>x_n>y_1>y_2>\cdots > y_n.\]
%\end{remark}

\begin{definition}
Let $G$ be a simple graph on $[n]$ vertices. Then, the binomial edge ideal corresponding to $G$ is denoted by $J_G$ and defined by 
\[ J_G = \big<f_{ij} = x_{i}y_{j}-x_{j}y_{i} \in R\mid \{i,j\}\in E(G)\,\, \text{with}\,\, i<j\big>.\]
\end{definition}
The notion of binomial edge ideals was introduced by Herzog et al., \cite{herzog2010binomial} and, independently, by Othani \cite{ohtani2011graphs}. %The binomial edge ideal can also be seen as a generalization of a connection between graphs with the ideals generated by $2\times 2$ minors of the generic matrix
%$\begin{bmatrix}
%    x_1 & x_2 & \cdots & x_n \\
%    y_1 & y_2 & \cdots & y_n 
%\end{bmatrix}.$
They proved that $J_G$ is a radical ideal over any field $K$, and they have described the minimal primes of $J_G$ combinatorially. Given any set $S\in \mathcal{C}(G)$, let 
 \[P_S = m_S + J_{\Tilde{G}_1} + J_{\Tilde{G}_2} + \cdots + J_{\Tilde{G}_{c(G\setminus S)}},\]
where $G_1, \ldots, G_{c(G\setminus S)}$ are the connected components of $G\setminus S$ and $\tilde{G}_i$ is the complete graph on $V(G_i)$, for all $i=1,\ldots, c(G\setminus S)$. Then $P_S(G)$ is a minimal prime of $J_G$ and every minimal primes are of this form \cite[Corollary 3.9]{herzog2010binomial}. For simplicity of notation, we use $P_S$ instead of $P_S(G)$ whenever $G$ is clear from the context.

\begin{definition}
    Let $u=u_0,u_1, \ldots, u_k=v$ be a path between two vertices $u,v \in V(G)$, where $u<v$. Then, this path will be an admissible path if 
    \begin{itemize}
        \item $u_i\neq u_j$, for $i\neq j$.
        \item for each $u_i, i\in [k-1]$, either $u_i> v$ or $u_i <u$.
        \item For any proper subset $\{u_{i_1},u_{i_2},\ldots, u_{i_r}\}$ of $\{u_1, \ldots, u_k\}$, the sequence $u, u_{i_1}, \ldots, u_{i_r}, v$ is not a path.
    \end{itemize}
\end{definition} 

\begin{proposition} \label{prop_bino_grobner}
\cite[Theorem 2.1]{herzog2010binomial} Let $G$ be a simple graph on $[n]$ and $R$ be equipped with the lexicographic order. %$<$ is the monomial order on $R$ \Cref{remark_2 }. 
Then the set \[\bigcup_{u<v} \Bigl\{ \big(\prod_{u_i>v} x_{u_i} \big) \big( \prod_{u_j <u}y_{u_j} \big) : u=u_0,u_1, \ldots, u_k=v ~\text{is an admissible path in $G$} \Bigr\}\] gives a reduced Gr$\ddot{o}$bner basis of $J_G$.
\end{proposition}

\begin{remark}\cite[Theorem 1.1]{herzog2010binomial}
For a finite simple graph $G$, $G$ is closed if and only if \{$f_{ij}$ ~:~ $\{i,j\} \in E(G)\text{ with } i<j\}$ is a Gr\"obner basis of $J_G$.
\end{remark}

%\begin{definition}
% Let $M$ be a graded $R$ module and $\textbf{F}$ is a minimal graded free resolution of $M$ as an $R$ module.
% \[\textbf{F:}~ 0 \rightarrow \bigoplus_j R(-j)^{b_{g,j}} \rightarrow \cdots \rightarrow \bigoplus_j R(-j)^{b_{1,j}} \rightarrow \bigoplus_j R(-j)^{b_{0,j}} \rightarrow M \rightarrow 0 \]
%\end{definition}
%Then the Castelnuovo-Mumford regularity of $M$ is defined as $\reg (M) = \mathrm{max}\{ j-i : b_{i,j} \neq 0 \}$
%For a nonzero graded ideal $I$, define $\alpha(I) = \min \{ \deg(f): f\in I \setminus \{0\}\}$. Then, an equigenerated graded ideal $I$ has linear resolution if and only if $\alpha (I) = \reg (I)$. We say that $I$ has linear powers if $I^k$ has a linear resolution for all $k\geq 1$.

%\begin{remark}
 %   Let $G_1, G_2 , \ldots, G_{c(G)}$ be the connected components of a graph $G$. Suppose $R_j = K[x_i,y_i: i\in V(G_j)]$ and $R = K[x_i,y_i:i\in V(G)]$, then 
 %   \[ \reg(\frac{R}{J_G}) = \sum_{j=1}^{c(G)} \reg(\frac{R}{J_G})\].
%\end{remark}

%should I write regularity lemma, and the symbol $c$.

\section{The $\mathrm{v}$-number of Cohen-Macaulay binomial edge ideals}\label{sec:closedcm}

We say that a graph $G$ is Cohen-Macaulay if $R/J_G$ is Cohen-Macaulay. The first study on the classification of Cohen-Macaulay binomial edge ideals was done by Ene, Herzog and Hibi in \cite{ene2011cohen}, where they classified all the Cohen-Macaulay closed graphs. 
%A graph $G$ is said to be \textit{closed} if $J_G$ has a quadratic Gr\"{o}bner basis. The structure of Cohen-Macaulay closed graphs has been described in \cite[Theorem 3.1]{ene2011cohen}, using maximal cliques of such graphs. 
Let $G$ be a Cohen-Macaulay closed graph, and $F_1, \ldots, F_t$ be the maximal cliques of $G$. Then $F_i$'s can be ordered in such a way that $F_i\cap F_{i+1}$ is singleton for all $i=1,\ldots,t-1$, and $F_i\cap F_j = \emptyset$ if $i<j$ and $j \neq i+1$, for all $i = 1,\ldots, t-1$. Let $F_i\cap F_{i+1} = \{v_i\}$ and $\Tilde{C}(G) = \{v_i: i\in [t-1]\}$. Each of these $v_i$ is a cut vertex of $G$, and any cut set of $G$ must be a subset of $\Tilde{C}(G)$. The graph $H$ in \Cref{fig:cmclosed} is an example of a Cohen-Macaulay closed graph.

\begin{figure}[H]
    \centering
\begin{tikzpicture}[line cap=round,line join=round,>=triangle 45,x=1.5cm,y=1cm]
%\clip(-1.350294306810591,0.24452001220682554) rectangle (10.087724678474153,4.058794129819358);
\draw (2.,3.)-- (1.,2.);
\draw (1.,2.)-- (2.,1.);
\draw (2.,1.)-- (3.,2.);
\draw (3.,2.)-- (2.,3.);
\draw (2.,3.)-- (2.,1.);
\draw (1.,2.)-- (3.,2.);
\draw (3.,2.)-- (5.,2.);
\draw (3.,2.)-- (4.,3.);
\draw (4.,3.)-- (5.,2.);
\draw (5.,2.)-- (7.,2.);
\draw (7.,2.)-- (8.,3.);
\draw (7.,2.)-- (8.,1.);
\draw (8.,1.)-- (9.,2.);
\draw (9.,2.)-- (8.,3.);
\draw (8.,3.)-- (8.,1.);
\draw (7.,2.)-- (9.,2.);
\draw (9.,2.)-- (11.,2.);
\draw (1.5,2.7) node[anchor=north west] {$F_1$};
\draw (3.8,2.7) node[anchor=north west] {$F_2$};
\draw (5.8,2.7) node[anchor=north west] {$F_3$};
\draw (7.5,2.7) node[anchor=north west] {$F_4$};
\draw (9.6,2.7) node[anchor=north west] {$F_5$};
\draw (2.8,1.9) node[anchor=north west] {$v_1$};
\draw (4.9,1.9) node[anchor=north west] {$v_2$};
\draw (6.8,1.9) node[anchor=north west] {$v_3$};
\draw (8.9,1.9) node[anchor=north west] {$v_4$};
\begin{scriptsize}
\draw [fill=black] (2.,3.) circle (1.5pt);
%\draw[color=black] (2.14,3.37) node {$A$};
\draw [fill=black] (1.,2.) circle (1.5pt);
%\draw[color=black] (1.14,2.37) node {$B$};
%\draw[color=black] (1.32,2.91) node {$f$};
\draw [fill=black] (2.,1.) circle (1.5pt);
%\draw[color=black] (2.14,1.37) node {$C$};
%\draw[color=black] (1.32,1.45) node {$g$};
\draw [fill=black] (3.,2.) circle (1.5pt);
%\draw[color=black] (3.14,2.37) node {$D$};
%\draw[color=black] (2.78,1.45) node {$h$};
%\draw[color=black] (2.78,2.91) node {$i$};
%\draw[color=black] (1.72,2.17) node {$j$};
%\draw[color=black] (2.04,1.85) node {$k$};
\draw [fill=black] (5.,2.) circle (1.5pt);
%\draw[color=black] (5.14,2.37) node {$E$};
%\draw[color=black] (4.04,1.85) node {$l$};
\draw [fill=black] (4.,3.) circle (1.5pt);
%\draw[color=black] (4.14,3.37) node {$F$};
%\draw[color=black] (3.78,2.45) node {$m$};
%\draw[color=black] (4.32,2.45) node {$n$};
\draw [fill=black] (7.,2.) circle (1.5pt);
%\draw[color=black] (7.14,2.37) node {$G$};
%\draw[color=black] (6.04,1.85) node {$p$};
\draw [fill=black] (8.,3.) circle (1.5pt);
%\draw[color=black] (8.14,3.37) node {$H$};
%\draw[color=black] (7.78,2.45) node {$q$};
\draw [fill=black] (8.,1.) circle (1.5pt);
%\draw[color=black] (8.14,1.37) node {$I$};
%\draw[color=black] (7.32,1.45) node {$r$};
\draw [fill=black] (9.,2.) circle (1.5pt);
%\draw[color=black] (9.14,2.37) node {$J$};
%\draw[color=black] (8.78,1.45) node {$s$};
%\draw[color=black] (8.78,2.91) node {$t$};
%\draw[color=black] (7.72,2.17) node {$a$};
%\draw[color=black] (8.04,1.85) node {$b$};
\draw [fill=black] (11.,2.) circle (1.5pt);
%\draw[color=black] (11.14,2.37) node {$K$};
%\draw[color=black] (10.04,1.85) node {$c$};
\end{scriptsize}
\end{tikzpicture}
\caption{A Cohen-Macaulay closed graph $H$}
    \label{fig:cmclosed}
\end{figure}
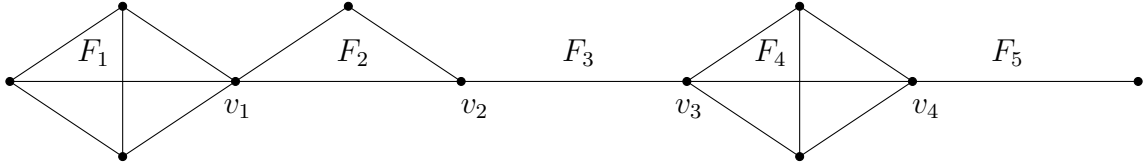

Our initial target is to give a combinatorial expression for the local $\mathrm{v}$-numbers of Cohen-Macaulay closed graphs. Using that, we then compute the $\mathrm{v}$-number of these graphs. %With that purpose, for every cut set $S$ of $G$, we will construct a homogeneous polynomial in $R$ whose degree will be equal to $\mathrm{v}_S(J_G)$.
Let $S$ be a cut set of $G$ and $P_S$ be the corresponding minimal prime of $J_G$. Suppose \[P_S = m_S + J_{\Tilde{G}_1} + J_{\Tilde{G}_2} + \cdots + J_{\Tilde{G}_{c(G\setminus S)}},\]
where $G_1, \ldots, G_{c(G\setminus S)}$ are the connected components of $G\setminus S$.
Since the induced subgraph on $\Tilde{C}(G)$ is a path and $S   \subseteq \Tilde{C}(G)$, for any $s\in S$, at most two cut vertices of $G$ can belong to $N_G(s) \setminus S$. Write $S= S_0 \cup S_1 \cup S_2,$ where
\begin{eqnarray*}
S_2 & = & \{ s\in S: \vert (N_G(s) \setminus S) \cap \Tilde{C}(G)\vert = 2 \},   \\ 
S_1 & = & \{ s\in S: \vert (N_G(s) \setminus S) \cap \Tilde{C}(G)\vert = 1 \}, \\
S_0 & = & \{ s\in S: \vert (N_G(s) \setminus S) \cap \Tilde{C}(G)\vert = 0 \}.
\end{eqnarray*}

For  $s \in S_2$, we set $(N_G(s) \setminus S) \cap \Tilde{C}(G) = \{v_{s_1}, v_{s_2} \}$ and for $s\in S_1$, set $(N_G(s) \setminus S) \cap \Tilde{C}(G) = \{v_s\}$.
Let $f_S = \prod_{s\in S_2} f_{v_{s_1} v_{s_2}}  \prod_{s\in S_1} f_{v_{s} s_j}  \prod_{s\in S_0} f_{s_i s_j},$
where $s_i, s_j \in (N_G(s) \setminus S)\setminus \Tilde{C}(G)$ and $\{v_{s_1}, v_{s_2}\}, \{v_{s},s_j\}, \{s_i,s_j\} \notin E(G)$ for every $s\in S$. %It is easy to observe that if $v_{s_i}, v_{s_j} \in N_{G\setminus S}(s) \cap \Tilde{C}(G)$, then they are in disjoint connected components of $G\setminus S$. 
Let $f_{D_S} = \prod_{v\in D_S}x_v $, where
\[ D_S = \{ v \in \Tilde{C}(G): v\in V(G\setminus S) ~\text{and}~ v \notin N_G(s) ~\text{ for any}~ s\in S \}.\]
\noindent
We illustrate the above notation in an example.
\begin{example}{\rm
Let us take two cut set $S= \{ v_2,v_4\}$ and $S' = \{ v_1,v_3,v_4\}$ of the graph $H$ in \Cref{fig:cmclosed}. Then for the cut set $S$, $S_2 = \{v_2\}$, $S_1 = \{v_4\}$ and $S_0 = \emptyset$. For $S'$, $S'_2 = \emptyset$, $S'_1 = \{ v_1, v_3\}$ and $S'_0 = \{ v_4\}$. Here $D_S = \emptyset$ and $D_{S'} =\emptyset$. If  $S''=\{v_2\}$, then $D_{S''} = \{v_4\}$. 
}
\end{example}

The following lemms helps us getting an upper bound for a local $\vv$-number.
\begin{lemma}\label{lemma_2}
Let $G$ be a Cohen-Macaulay closed graph and $f_S, f_{D_S}$ be defined as above. Then, \[ (J_G:f_Sf_{D_S}) = P_S.\]
\end{lemma}

\begin{proof}
To prove the required equality, it is sufficient to show that $f_Sf_{D_S}\notin P_S$ and $P_S(f_Sf_{D_S})\subseteq J_G$.
By our construction, for each $s \in S$, none of $f_{v_{s_1}v_{s_2}}$, $f_{v_{s}s_j}$ or $ f_{s_is_j}$ belongs to $P_S$. If $v\in D_S$, then $x_v\notin P_S$ as $v\notin S$, which implies that $f_{D_S} \notin P_S$. Hence, $f_Sf_{D_S} \notin P_S.$

Now, we need to prove that $P_S(f_Sf_{D_S})\subseteq J_G$. This is equivalent to showing that for any generator $g$ of $P_S$, $gf_{S}f_{D_S}$ belongs to $J_G$. For any $s\in S$, if $s_i,s_j\in N_G(s)$ are such that $\{s_i,s_j\} \notin E(G)$, then $x_sf_{s_is_j}\in J_G$ and $y_sf_{s_is_j} \in J_G$. Hence, $x_sf_S, y_sf_S \in J_G ~\text{for all}~ s\in S.$ Now, suppose $g = f_{lk} \in J_{\Tilde{G}_i}$ for some $i$ such that $f_{lk} \notin J_G$, i.e., $\{l,k\} \notin E(G)$. From the structure of $G$ it follows that there must be a path $l,v_{i_1},v_{i_2}\ldots, v_{i_m},k$, where $v_{i_r}\in \Tilde{C}(G)\cap G_i$ for all $r \in [m]$. Then we know that 
\[ \left(\prod_{r=1}^m t_{v_{i_r}}\right)f_{lk} \in J_G, \text{where}~ t_{v_j}\in \{x_{v_j},y_{v_j}\}\]

\noindent Note that for each $r\in [m]$, $v_{i_r}$ either belongs to $D_S$ or belongs to $N_G(s) \setminus S$ for some $s\in S_1 \cup S_2$. If $v_{i_r} \in N_G(s) \setminus S$ and $s\in S_2$, then $v_{i_r} \in \{v_{s_1}, v_{s_2} \}$, and if $s\in S_1$, then $ v_{i_r} = v_s$. Let
\[ h=\prod_{v_{i_r}\in D_S}x_{v_{i_r}}  \prod_{v_{i_r}\in N_G(s) \setminus S ~\&~ s \in S_2}f_{{v_{s_1}} v_{s_2}}  \prod_{v_{i_r}\in N_G(s) \setminus S ~\&~ s\in S_1  } f_{v_s s_j}.\]
Here, $s_j \in N_G(s) \setminus S$ is the same neighbour vertex of $s$ that we choose in the construction of the polynomial $f_S$. By our construction, $h$ divides $f_Sf_{D_S}$. Each term of $h$ is of the form $(\prod_{r=1}^m t_{v_{i_r}})z$, where $z$ is a monomial. Hence, 
$hf_{lk} \in J_G$ which implies $ f_Sf_{D_S}f_{lk} \in J_G.$
This completes the proof.
\end{proof}
To proceed further, the use of the initial ideal is going to play a crucial role. The following proposition is one of the key tools in the proof of most of the later theorems. 

\begin{proposition}\label{prop_1}
Let $G$ be a simple graph and $f\in R_m$ be such that $(J_G:f)= P_S$ for some $P_S \in \mathrm{ Min}(J_G)$. Then there exists $g\in R_m$ so that $\ini_{<}(g) \notin \ini_{<}(J_G)$ and $(J_G:g) = P_S$.
\end{proposition}

\begin{proof}
Since $(J_G:f)= P_S$, $f\notin J_G$. Now, if $\ini_{<}(f) \notin \ini_{<}(J_G)$, then take $g=f$.
If $\ini_{<}(f) \in \ini_{<}(J_G)$, then there exists $f_1\in J_G$, which is homogeneous and $\ini_{<}(f_1) = \ini_{<}(f)$. If $g_1 = f - f_1$, then $\ini_{<}(f) > \ini_{<}(g_1)$ and $(J_G:g_1)= P_S$. If $\ini_{<}(g_1)\notin \ini_{<}(J_G)$, then take $g = f-f_1$. If $\ini_{<}(g_1)\in \ini_{<}(J_G)$, then again choose a homogeneous polynomial $f_2 \in J_G$ such that $\ini_{<}(f_2) = \ini_{<}(g_1)$. In this case, we take $g_2 = g_1 - f_2 = f- f_1 -f_2$. Then $\ini_{<}(g_2)< \ini_{<}(g_1) <\ini_{<}(f)$ and $(J_G:g_2)= P_S$. This process stops after a finite number of stages as the number of monomials less than $\ini_{<}(f)$ is finite. Let us assume that the process stops after $k$ many steps. Then we get $f_1,f_2,\ldots f_k \in J_G$ for some $k>0$, such that each $f_i\in R_m$, $g = f - f_1 - f_2 - \cdots -f_k \in R_m$ and $\ini_{<}(g) \notin \ini_{<}(J_G)$ as $f\notin J_G$. Hence, we obtain the required $g$ with $(J_G:g) = P_S$.
\end{proof}

The following theorem provides a combinatorial expression of local $\mathrm{v}$-numbers of closed Cohen-Macaulay binomial edge ideal. We fix the labelling on $V(G)$ such that $G$ is closed, and on $R$, we take lexicographic ordering.

\begin{theorem} \label{thm_1}
Let $G$ be a Cohen-Macaulay closed graph. Then for any cut set $S$ of $G$, $\mathrm{v}_S(J_G) = 2|S| + |D_S|$.
    
\end{theorem}

\begin{proof}
The degree of the polynomial $f_Sf_{D_S}$ is $2|S|+|D_S|$. So, from \Cref{lemma_2}, we get
\[\mathrm{v}_S(J_G) \leq 2|S| + |D_S|. \]
Now, we prove the reverse inequality. Let us assume that $(J_G:f) = P_S$, for some homogeneous polynomial $f$ of degree $d$. We will show that $d\geq 2|S| + |D_S|$. Due to \Cref{prop_1}, we can assume that $\ini_{<}(f) \notin \ini_{<}(J_G)$.
Let $v_i\in D_S$, recall $\{v_i\} = F_i \cap F_{i+1}$. Then we choose two vertices $a_1, a_2$ from $F_i$ and $F_{i+1}$, respectively, which are different from $v_i$. Additionally, $a_1 \neq v_{i-1}$, if $|F_i| \geq 3$ and $a_2 \neq v_{i+1}$, if $|F_{i+1}| \geq 3$. Clearly $\{a_1,a_2\} \notin E(G)$, but $f_{a_1a_2} \in P_S$. Since $(J_G:f) = P_S$, 
\[f  \in  (J_G: f_{a_1a_2}) = \bigcap_{P_T \in Min(J_G), f_{a_1a_2} \notin P_T} P_T  = \bigcap_{T\in \mathcal{C}(G), v\in T} P_T = m_{\{v\}} + J_{G \setminus \{v\}},
\]
where the last equality follows from
\cite[Proposition 5.2]{bolognini2021cohen}.
%Here, $G\setminus \{v\}$ is also a closed graph. 
Since $\ini_{<} (f)\notin \ini_{<} (J_G)$, $\ini_{<} (f)\notin \ini J_{G \setminus \{v\}}$ as  well. Therefore, $\ini_{<} (f) \in m_{\{v\}}$. Hence $t_v$ divides $\ini_{<} (f)$ for each $v \in D_S$, where $t_v \in \{x_v, y_v\}$.
%
%This is true for all $v \in D_S$. Therefore
%\[ \ini f \in \bigcap _{v\in D_S} m_{\{v\}}.\]
%Then there must be a monomial of the form $\prod_{i\in D_S} t_i$, where $t_i\in \{x_i,y_i\}$ which divides $inf $. 
Therefore $\prod_{v\in D_S} t_v$ divides $\ini_{<} (f)$.% In particular, $d \geq |D_S|.$

Since $(J_G : f) = P_S$, for any $s\in S$, $f\in (J_G:x_s)= J_{G_s}$ \cite[Proposition 3.1]{ambhore2023v}.
Recall that $F_i$'s are the maximal cliques of $G$, where $i\in[t]$, and for every $s\in S$, $\{s\} = F_j \cap F_{j+1}$, for some $j\in [t-1]$. Then the graph $G_s$ has $t-1$ maximal cliques which are $F_1, \ldots, F_{j-1}, F_j \cup F_{j+1}, F_{j+2}, \ldots, F_t$. Since $G$ is closed, it follows from \cite[Theorem 2.2]{ene2011cohen} that $G_S$ is closed. Therefore,
\[ \ini_{<}(J_{G_s}) = \ini_{<} (J_G) + ( x_{s_i}y_{s_j} : s_i<s_j, \{s_i,s_j\} \subseteq N_G(s)\text{ and } \{s_i,s_j\}\notin E(G)).\]
Since $\ini_{<} (f) \notin \ini_{<} (J_G)$ and $\ini_{<} (f) \in \ini_{<} (J_{G_s})$, for each $s\in S$,
$x_{s_i}y_{s_j}$ divides $\ini_{<} (f)$ for some $\{s_i,s_j\} \subseteq N_G(s)$ with $\{x_{s_i},x_{s_j}\} \not\in E(G)$.  
Suppose $s, s' \in S$ such that $s< s'$ and $\{s_i,s_j\}\subseteq N_G(s)$, $\{s_i',s_j'\}\subseteq N_G(s')$ with $\{s_i,s_j\}, \{s_i',s_j'\}\not\in E(G)$. Since $G$ is a Cohen-Macaulay closed graph, it follows from \cite[Theorem 3.1]{ene2011cohen} that $s\in F_r$ and $s' \in F_{r+1}$, for some $r\in [t-1]$ such that  $s_i<s_j\leq s_i' < s_j'$. Therefore, $\gcd(\ini_{<} (f_{s_is_j}), \ini_{<} (f_{s_i's_j'})) = 1$. % So, in such situation only $x_{s_i}y_{s_j}$ and $x_{s_i'}y_{s_j'}$, corresponding to the indices $s_i,s_j$ and $s_i',s_j'$, can divide $\ini f$. 
Thus, \[ \prod _{s\in S} x_{s_i}y_{s_j} | \ini_{<} (f).\] 
%So, \[ d \geq 2|S|.\]
Since $D_S \cap N_G(S) = \emptyset,$, it follows that $\prod_{v\in D_S} t_v\cdot \prod _{s\in S} x_{s_i}y_{s_j}$ divides $\ini_{<} (f).$ Hence,
 $d\geq 2|S| + D_S$. Since $f$ was chosen arbitrarily, we have $\vv_S(J_G)\geq 2|S|+|D_S|$. In particular, $\vv_S(J_G)= 2|S|+|D_S|$.
\end{proof}

Now, we are ready to prove the main theorem of this section.

\begin{theorem} \label{thm_2}
Let $G$ be a Cohen-Macaulay closed graph with $t$ maximal cliques. Then \[\mathrm{v}(J_G) = (t-1) - \big\lfloor \frac{t-1}{3}\big\rfloor =\big\lceil\frac{2(t-1)}{3}\big\rceil \]
\end{theorem}

\begin{proof}
From \Cref{thm_1}, $\mathrm{v}_S(J_G) = 2|S| + |D_S|$, for every $S\in \mathcal{C}(G)$. We now identify a cut set $S_0$ such that $\mathrm{v}_{S_0}(J_G)$ is the minimum. By construction,
\[|D_S| = |\Tilde{C}(G)| - |S| - |N_S|,\]
where $N_S = (N_G(S) \setminus S) \cap \tilde{C}(G)$. %\avj{Do we really need to write what is $\Tilde{C}(G)$ here?}
Therefore,
\begin{eqnarray*}
\vv_S(J_G) & = & 2|S| + |\Tilde{C}(G)| - |S| - |N_S| \\
& = & |\Tilde{C}(G)| + |S| - |N_S| \\
& = & (t-1) + |S| - |N_S|.
\end{eqnarray*}
To get a minimum value of the above expression, $|S|-|N_S|$ needs to be minimum. Note that for every $s\in S$, there can be at most two cut vertices which are neighbors of $s$, i.e., $|N_S| \leq 2|S|$. Since $G$ is Cohen-Macaulay, $\vv_\emptyset(J_G) = (t-1)$. 

\begin{enumerate}
    \item Suppose $t\geq 4$. Let $\mathcal{P}$ denote the collection of the sets $S\in \mathcal{C}(G)$ such that $|N_S| = 2|S|$. We claim that $\mathcal{P} \neq \emptyset$. Then for every $S\in \mathcal{P}$, $\vv_S(J_G) = (t-1) - |S|$. So, our goal is to find out a cut set $S_0$ of maximum cardinality from $\mathcal{P}$, because then $\vv_{S_0}(J_G) = (t-1) - |S_0|$ is going to be the minimum of all $\vv_S(J_G)$. Recall the notation $\Tilde{C}(G) = \{v_i: \{v_i\} = F_i \cap F_{i+1}, i\in [t-1]\}$. A cut set $S = \{v_{i_1}, v_{i_2}, \ldots, v_{i_r}\}$ with $i_j <i_{j+1}$, belongs to $\mathcal{P}$ if and only if $i_1 >1$, $i_r < (t-1)$ and $i_{j+1} - i_j \geq 3$. We construct $S_0$ in such a way that ${i_1} = 2$ and $i_{j+1} - i_j = 3$. Let 
    \[S_0=\{v_k ~:~ k=3i+2, i\geq 0, k < t-1\}.\]
    %such that the difference between two consecutive $v_i$'s in $S_0$ is three and $v_1, v_{t-1}$ does not belongs to $S_0$. 
    So, $S_0 \in \mathcal{P}$ and by the construction of $S_{0}$, this is a set with maximum cardinality in $\mathcal{P}$. %To compute $|S_0|$ easily, we rename the elements of $\Tilde{C}(G)$. Let $v_i = u_{i+1}$, then
    %\[ \Tilde{C}(G) =\{ u_2,u_3,\ldots ,u_t\}\]
    %and 
    %\[ S_0= \{ u_3,u_6,u_9,\ldots , u_{3 \times \lfloor \frac{ t-1}{3}\rfloor} \}\] 
    %,as $u_t \notin S_0$, we will stop at $u_i$, where $i$ is largest multiple of $3$ less than equal to $t-1$, which is $3 \times \lfloor \frac{ t-1}{3}\rfloor$. 
    Clearly, $|S_0| = \lfloor \frac{ t-1}{3}\rfloor$.
    Hence,
    \[\mathrm{v}(J_G) = \mathrm{v}_{S_0}(J_G) = (t-1) - \bigg\lfloor \frac{t-1}{3}\bigg\rfloor=\bigg\lceil \frac{2(t-1)}{3}\bigg\rceil\]

  \item If $t=2,3$, then  $|N_S| \leq 1$ for all non-empty $S\in \mathcal{C}(G)$. This implies $|S| - |N_S| \geq 0$. So $\mathrm{v}_S(J_G) \geq t-1$. Since $\vv_{\emptyset}(J_G) = t-1 $, we have
\[ \mathrm{v}(J_G) = \vv_{\emptyset}(J_G) = t-1 =\bigg\lceil \frac{2(t-1)}{3}\bigg\rceil. \]

  \item If $t=1$, i.e., $G$ is a complete graph, then 
  \[ \vv(J_G) = 0 = \bigg\lceil \frac{2(t-1)}{3}\bigg\rceil. \]
  \end{enumerate}
\end{proof}

\begin{corollary}\label{cor_path_v}
     Let $P_n$ be the path graph on $n$ vertices. Then $\mathrm{v}(J_{P_n}) = \big\lceil\frac{2(n-2)}{3}\big\rceil.$
\end{corollary}

\begin{proof}
By \cite{herzog2010binomial}, $P_n$ is a Cohen-Macaulay closed graph with the number of maximal cliques equal to $n-1$. Hence, the assertion follows from \Cref{thm_2}.
\end{proof}
For an arbitrary graph $G$, $\vv_\emptyset (J_G)$ is expressed in terms of the connected domination number of $G$ in \cite{ambhore2023v} and \cite{jaramillo2023connected}. Computing the connected domination number of a graph is an NP-complete problem even for the subclass like bipartite graphs. However, for a Cohen-Macaulay graph $G$, we give a simpler combinatorial expression for $\vv_\emptyset (J_G)$ as follows.

\begin{proposition}\label{prop_CM_v} 
    Let $G$ be a Cohen-Macaulay graph and $r$ be the number of cut vertices of $G$. Then $\vv_\emptyset(J_G) = r$.
\end{proposition}

\begin{proof}
Since $G$ is Cohen-Macaulay, $G$ is accessible \cite[Theorem 3.5]{bolognini2021cohen}. It follows from \cite[Theorem 3.6] {ambhore2023v} that $\vv_{\emptyset}(J_G) = \text{$\mathrm{min}$}\{\deg(m): m ~ \text{is a generator of}~ \bigcap_{S\in \mathcal{C}(G)} m_S\}$. Since $G$ is accessible, every cut set contains a cut vertex \cite[Definiton 2.2]{bolognini2021cohen}. Thus, $\bigcap_{S\in \mathcal{C}(G)} m_S = \bigcap_{\{v\} \in \mathcal{C}(G)}m_{\{v\}}$, which implies that the minimum degree of a generator of $\bigcap_{S\in \mathcal{C}(G)} m_S$ is the number of cut vertices of $G$. Hence, $\vv_\emptyset(J_G) = r$. 
\end{proof}

\begin{remark}{\rm
The converse of the above proposition is not correct. Let us take $G$ to be a star graph with at least four vertices, and $x$ be the only vertex of $G$ with $\deg_G(x) > 1$. Then $G$ is a cone graph. So, according to \cite[Theorem 3.20]{ambhore2023v},$\vv(J_G) = \vv_\emptyset (J_G) = 1 = |\{u\in V(G)~:~ \text{ $u$ is a cut vertex}\} |$. But It follows from \cite[Lemma 3.1]{herzog2010binomial} $J_G$ is not unmixed, so, $G$ is not Cohen-Macaulay. 
   }
\end{remark}

\section{$\mathrm{v}(J_G)=2$ and expected $\vv$-number of cycles and binary trees%on $\mathrm{v}(J_G)$
}\label{sec:v=2}

In this section, we first characterize all graphs whose binomial edge ideals have $\vv$-number equal to two. Then, we discuss the behaviour of the $\vv$-number for decomposable graphs. Also, we give a sharp upper bound of the $\vv$-number for binary trees and cycles.\par 

Since $\vv$-number is additive on connected components, \cite[Corollary 3.12]{ambhore2023v}, it is enough to consider only connected graphs for the classification of binomial edge ideals with $\vv$-number equal to two.

\begin{theorem} \label{thm_3}
Let $G$ be a connected graph. Then $\mathrm{v}(J_G) = 2$ if and only if $G$ is not a cone graph and satisfies one of the following conditions:

\begin{enumerate}
    \item There exist vertices $u, v$ with $\{u,v\} \in E(G)$ such that $N_G(u) \cup N_G(v) = V(G)$.

    \item  There exist vertices $u, v$ with $\{u,v\} \notin E(G) $ such that $N_G(u)\cap N_G(v)$ is a non-empty cut set of $G$ that disconnects $u$ and $v$. If $G_i,G_j \in C(G\setminus (N_G(u)\cap N_G(v)))$ are components containing $u$ and $v$ respectively, then $G_i = \text{Cone}(u,G_i\setminus \{u\})$, $G_j=\text{Cone}(v,G_j\setminus \{v\})$ and all other connected components of $G\setminus (N_G(u)\cap N_G(v))$  are complete graphs.
\end{enumerate}
\end{theorem}

\begin{proof}Let $\vv(J_G) = 2$. 
%    \begin{enumerate} 
First assume that $\mathrm{v}(J_G) = \mathrm{v}_\emptyset (J_G) = 2$.
Then, by \cite[Theorem 3.20]{ambhore2023v}, $G$ is not a cone graph. Let $\{u,v\}$ be a minimal completion set (equivalently, a minimal connected dominating set) of $G$. Then by the definition of completion set, $\{u,v\}\in E(G)$ and $N_G(u) \cup N_G(v) = V(G)$. Hence $G$ satisfies $(1)$. %\avj{In this part, the existence of $u, v$ is not very clear to me. In the next part, what are $u$ and $v$ comes out clearly in the proof. In this part also it should be clearly mentioned. That will also make it clear why it is an edge in $G$.}\textcolor{blue}{I have added the reason}

Suppose $\mathrm{v}(J_G) =2$ and $\mathrm{v}_\emptyset(J_G) > 2$. In this case, we show that $G$ satisfies $(2)$.
Since $\vv(J_G) = 2$, there exist $f\in R_2$ and a non-empty cut set $S$ such that $( J_G: f) = P_S.$
By \Cref{prop_1}, we may assume that $\ini_{<} (f) \notin \ini_{<} (J_G)$. For each $s\in S$,
$f\in (J_G:x_s) = J_{G_s}$ so  that $\ini_{<} (f) \in \ini_{<}(J_{G_s})$.
Since $\ini_{<} (f)$ has degree 2, it must be of the form $x_uy_v$, where $\{u,v\} \in E(G_s) \setminus E(G)$. Therefore  $u,v\in \cap_{s\in S} N_G(s)$, and hence $S\subseteq N_G(u) \cap N_G(v)$.
If $u$ and $v$ are in the same connected component of $G\setminus S$, then $f_{uv} \in P_S$. This implies that $f_{uv}f \in J_G$. Since $x_uy_v$ divides both the initial terms, we get $x_u^2y_v^2 \in \ini_{<} (J_G)$. Hence, $x_uy_v \in \ini_{<} (J_G)$ as  $\ini_{<} (J_G)$ is radical and this contradicts the assumption that $\{u,v\} \notin E(G).$ Therefore, $u$ and $v$ are in distinct connected components of $G\setminus S$. This proves that $S = N_G(u) \cap N_G(v).$  Now, let $G_i \in C(G\setminus S)$ be the component containing $u$. For any $l\in V(G_i)$, if $u \neq l$, then $f_{lu} \in P_S$.  Proceeding as in the previous case, we get,
\[x_ly_ux_uy_v \in \ini_{<} (J_G) \hspace{0.3cm} \text{or} \hspace{0.3cm} x_uy_lx_uy_v \in \ini_{<} (J_G) , \]
depending on $l<u$ or $u<l$ respectively. Then there must be a generator of $\ini_{<}({J_G})$ that divides one of the above monomials, i.e., there exists an admissible path in $G$ on the vertex set $\{u,v,l\}$.
Since $v\notin G_i$ and $\{u,v,l\}\cap S=\emptyset$, the only possible admissible path among the vertices $u,v$, and $l$ is the edge $\{u,l\}$. So, $\{u,l\} \in E(G)$, for all $l\in V(G_i)$. Hence $G_i = \text{Cone} ( u, G_i\setminus \{u\})$. Similarly if $G_j \in C(G\setminus S)$ is such that $v \in V(G_j)$, then $G_j = \text{Cone}(v, G_j\setminus \{v\})$.
Let $G_r \in C(G\setminus S)$ be a component not containing $u$ or $v$. Then for any $l,k \in V(G_r) $, such that $l<k$, $f_{lk}\in P_S$. Therefore $x_ly_kx_uy_v\in \ini_{<} (J_G)$.
Since $\{u,v\} \notin E(G)$ and $\{l,k\} \cap \{u, v\} = \emptyset$, the only possible generator of $\ini_{<} (J_G)$ that divides the above monomial is $x_ly_k$. This conclude that $\{l,k\} \in E(G)$, i.e., $G_r$ is complete. Hence, $G$ satisfies $(2)$.

Conversely, suppose $G$ is not a cone graph and it satisfies $(1)$, then $\vv(J_G) \geq 2$ by \cite[Theorem 3.20]{ambhore2023v} and $\{ u,v\}$ gives a  minimal completion set of $G$. Thus, it follows from \cite[Theorem 3.6]{ambhore2023v} that $\vv (J_G) = \vv_{\emptyset} (J_G) = 2$. Now suppose $G$ has the structure as described in $(2)$, then we will get 
\[ (J_G : f_{uv}) = P_S,\]
where $S = N_G(u) \cap N_G(v)$ is a non-empty cut set of $G$. Since $G$ is not a cone graph, $\mathrm{v}(J_G) = 2$ again by \cite[Theorem 3.20]{ambhore2023v}.
\end{proof}

%Now, we show that $\vv_\emptyset$ is additive describe the behavior of $\vv_\emptyset$, under the graph operation gluing of two graphs, along a common free vertex. 
Now we describe the behavior of $\vv_\emptyset$ on clique-sum over a free vertex.
Let $G = G_1 \cup_v G_2$ be the clique-sum of $G_1$ and $G_2$, where $V(G_1)\cap V(G_2) =\{v\}$ is a free vertex of both $G_1$ and $G_2$. Then it follows from \cite[Lemma 2.3]{rauf2014construction} that $\mathcal{C}(G)= \mathcal{A}\cup \mathcal{B}$, where
\[ \mathcal{A}= \{T_1\cup T_2: T_i\in \mathcal{C}(G_i), ~\text{for}~ i= 1,2\},\]
\[ \mathcal{B} = \{T_1\cup T_2 \cup \{v\} : T_i\in \mathcal{C}(G_i) ~\text{and}~ T_i \notin N_{G_i}[v] \}.  \]
%\avj{Question: If $G_1 \cap G_2 = K_r$, then do we have a similar result? If yes, then can it give a more general result for the $\vv$-number?}

\begin{proposition} \label{prop_decom}
Let $G = G_1\cup_v G_2$ be such that $V(G_1)\cap V(G_2)= \{v\}$, where $v$ is a free vertex of $G_1$ and $G_2$. Then $\mathrm{v}_\emptyset(J_G) = \mathrm{v}_\emptyset(J_{G_1}) + \mathrm{v}_\emptyset(J_{G_2})+1$.  
\end{proposition}

\begin{proof}
We know $\mathrm{v}_\emptyset(J_G) = \text{min}\{\deg(m): m ~ \text{is a generator of}~ \bigcap_{S\in \mathcal{C}(G)} m_S\} $ \cite[Theorem 3.6]{ambhore2023v}. From the above description of cut sets of $G$, we get 
\[ \bigcap _{S\in \mathcal{C}(G)} m_S = (\bigcap_{S_1\in \mathcal{C}(G_1)}m_{S_{1}}) \cap(\bigcap _{S_2\in \mathcal{C}(G_2)}m_{S_{2}}) \cap m_{v}\]
Since $v$ is a free vertex, $v\notin S_i$ for any $S_i\in \mathcal{C}(G_i)$, where $i=1,2$. Therefore, $\cap_{S_{1}\in \mathcal{C}(G_1)}m_{S_{1}}$, $\cap _{S_{2}\in \mathcal{C}(G_2)}m_{S_{2}}$ and $m_{v}$ are monomial ideals whose minimal generating sets are made of different sets of variables. Thus, any minimal generator of $\cap _{S\in \mathcal{C}(G)} m_S$ is a product of minimal generators from each of the above three ideals. Hence, $\mathrm{v}_\emptyset(J_G) = \mathrm{v}_\emptyset(J_{G_1}) + \mathrm{v}_\emptyset(J_{G_2})+1$.
\end{proof}

\begin{remark}\label{remark_1}
Due to  \Cref{prop_decom}, one may ask whether $\mathrm{v}(J_G) = \mathrm{v}(J_{G_1}) + \mathrm{v}(J_{G_2})+1$ for a decomposable graph $G=G_{1}\cup_{v} G_{2}$. In general, $\mathrm{v}(J_G) \neq \mathrm{v}(J_{G_1}) + \mathrm{v}(J_{G_2})+1$. %For path graph $P_n$ we have proved that $\mathrm{v}(J_{P_n}) = (n-2) - \lfloor \frac{n-2}{3}\rfloor$. 
For example, we consider $P_6$, which is a decomposable graph into $P_4$ and $P_2$. It is easy to verify that $\mathrm{v}(J_{P_4}) = 2$, $\mathrm{v}(J_{P_3})= 1$ and $\mathrm{v}(J_{P_6}) = 3\neq \mathrm{v}(J_{P_4})+\mathrm{v}(J_{P_3})+1$. Again, it was proved in \cite{jayanthan2019regularity} that $\reg(R_{1}/J_{G_1})+\reg(R_{2}/J_{G_2})=\reg(R/J_G)$, where $R_i=K[\{x_i,y_i : i\in V(G_i)\}]$ for $i=1,2$. So, one may expect this for $\vv$-number also. But, note that $P_3$ is decomposable into two $P_2$, and $\vv(J_{P_3})=1\neq 2\vv(J_{P_2})$ as $\vv(J_{P_2})=0$.
\end{remark}

In general for any graph $G$, $\vv_\emptyset(J_G)$ gives a upper bound of $\vv(J_G)$, but for many graph classes, this bound is far from the $\vv$-number. In the next two theorems, we give a better upper bound of $\vv$-number of binomial edge ideals, for binary tree and cycle graphs. We begin with binary tree graphs. Let us fix some notations first. 
\noindent

 %\avj{I don't think we should discuss the construction of $B_n$ here. This is standard and reader is expected to know.} %We construct $B_n$ recursively. In the initial step or $0$th step, we choose a vertex and denote it as $B_0$. In the next step ($1$st step), we attach two whiskers with the vertex and construct a cone graph which is $B_1$. So, $B_n$ is constructed at the $n$th step of this process, by attaching two whiskers with every degree one vertices of $B_{n-1}$, for $n\geq 2$. (Pic)
 %In the next theorem, We will give an upper bound of $\mathrm{v}(J_{B_n})$, using $\mathrm{v}(J_{B_{n-3}})$, recursively. 
For $n \geq 0$, let $B_n$ denote the binary tree on $2^{n+1}-1$ vertices and $R^{(n)} = K[x_i,y_i: i\in V(B_{n})]$ be the polynomial rings corresponding to the graphs $B_n$. Let us write $R^{(n)}= \oplus_{d=0}^{\infty} R_d^{(n)}$, where $R_d^{(n)}$ denotes the $K$-vector space of all homogeneous polynomials of degree $d$ with respect to the standard grading on $R^{(n)}$. For all $n\geq 0$, the binary tree $B_n$ has $2^n$ pendant vertices and we denote them by $n_i$, where $i\in [2^n]$. So, $V(B_n) = \bigcup_{k\in [n] \cup \{0\} } \bigcup_{i \in [2^k]} k_i$.

\begin{theorem} \label{thm_6}
Let $B_n$ be a binary tree. Then $\mathrm{v}(J_{B_0}) =0 $, $\mathrm{v}(J_{B_1}) = 1, \mathrm{v}(J_{B_2}) = 2$ and $\mathrm{v}(J_{B_n}) \leq 2^{n-1} + \mathrm{v}(J_{B_{n-3}})$ for all $n \geq 3$.
\end{theorem}

\begin{proof}
Since $B_0$ is a vertex, $\vv(J_{B_0}) = 0$, and $B_1$ being a non-complete cone graph $\vv(J_{B_1}) = 1$ by \cite[Theorem 3.20]{ambhore2023v}. Since $B_2$ satisfies \Cref{thm_3}(2), it follows that $\vv(J_{B_2}) = 2$. Now, we assume that $n \geq 3$. Let $S \in \mathcal{C}(B_{n-3})$ be such that $\mathrm{v}(J_{B_{n-3}}) = \mathrm{v}_S(J_{B_{n-3}}) = d$ and $f\in R_d^{(n-3)}$ with $(J_{B_{n-3}}: f) = P_S(B_{n-3})$. 
%From the above discussion, $B_{n-3}$ is an induced subgraph of $B_n$. 
Let $S' = \{(n-2)_i: i\in [2^{n-2}]\} \subseteq V(B_n)$ be the collection of the pendant vertices of $B_{n-2}$. Then, $S \cup S'$ is a cut set of $B_n$ and
    \[C(B_n\setminus (S\cup S')) = C(B_{n-3} \setminus S) \sqcup_{i\in [2^{n-1}]} H_i,\]
where $H_i \cong P_3$ for each $i$. Now, every $s_j' \in S'$ is adjacent to the degree $2$ vertex of $H_i$ for precisely two $i$'s. Let us denote those vertices by $s'_{j1}$ and $s'_{j2}$. Let 
$g = f \cdot \prod_{s_j' \in S'} f_{s'_{j1} s'_{j2}} \in R^{(n)}$.
\smallskip

\noindent\textbf{Claim:} $P_{S\cup S'}(B_n) = (J_{B_n}:g)$.
   \smallskip
 
\noindent\textit{Proof of the claim.} First note that $P_{S\cup S'}(B_n) =P_S(B_{n-3}) + m_{S'} + \sum_{i\in [2^{n-1}]} J_{\tilde{H}_i}$, where $\tilde{H}_i$ is the complete graph on $V(H_i)$. Since $(J_G: f) = P_S(B_{n-3})$, we have
\[ P_S(B_{n-3})f \subseteq J_{B_{n-3}} \subseteq J_{B_n}. \]
    
\noindent Since $s'_{j1}$ and $s'_{j2} \in N_{B_n}(s_j')$, for every $s_j' \in S'$, $x_{s_j} f_{s'_{j1} s'_{j2}} \in J_{B_n}$ and $y_{s_j} f_{s'_{j1} s'_{j2}} \in J_{B_n}$. Hence,
   \[  m_{S'}\prod_{s_j' \in S'} f_{s'_{j1} s'_{j2}} \subseteq J_{B_n}. \]

\noindent For every $f_{lk} \in J_{\tilde{H}_i}$ with $f_{lk} \notin J_{B_n}$, there exists $s_j' \in S'$ such that either $s'_{j1}$ or $s'_{j2}$ is a neighbor of both $l$ and $k$. Then, we have
\[ f_{lk} f_{s'_1 s'_2} \in J_{B_n}.\]
Hence, $P_{S\cup S'}(B_n) \subseteq (J_{B_n}:g)$. Now, $f\in R_{d}^{(n-3)}$ and $f \notin P_S(B_{n-3})$ implies that $f \notin P_{S\cup S'}(B_n)$. Again, for every $s_j' \in S'$, $s'_{j1}$ and $s'_{j2}$ are in two different connected components of $G\setminus S'$. Hence $f_{s'_{j1} s'_{j2}} \notin P_{S\cup S'}(B_n)$. This concludes that $(J_{B_n} : g) = P_{S\cup S'}(B_n)$. Therefore,
   \[ \mathrm{v}(J_{B_n}) \leq \mathrm{v}_{S\cup S'}(J_{B_n}) \leq \deg(g) = \deg(f) + 2|S'| = \mathrm{v}(J_{B_{n-3}}) + 2^{n-1}.  \]
\end{proof}
       
\begin{corollary}\label{cor_binary_v}
    As a consequence of the above theorem, we have,
\begin{equation*}
 \mathrm{v}(J_{B_n}) \leq \begin{cases}
    2^{n-1}+2^{n-4}+ \cdots + 2^2 & \text{if  $n \cong 0 (mod ~3)$}.\\
    2^{n-1}+ \cdots+ 2^3+1 & \text{ $n \cong 1 (mod ~3)$}.\\
    2^{n-1} + \cdots + 2^4 +2 & \text{ $n \cong 2 (mod ~3)$}.\\
  \end{cases}
\end{equation*}
\end{corollary}

Now we focus on the cycle graphs. For every cut set of the cycle graphs, we construct a polynomial whose degree is going to be the local $\vv$-number. Let $C_n$ be a cycle graph with $n$ vertices, $S$ be a cut set of $C_n$, and $$P_S = m_S + J_{\Tilde{H}_1} + J_{\Tilde{H}_2} + \cdots + J_{\Tilde{H}_{c(C_n\setminus S)}}$$ be the minimal prime of $J_{C_n}$ corresponding to $S$.
Here $H_i$'s are the connected components of $C_n\setminus S$.
Then for every $s\in S$, there are exactly two neighbour vertices of $s$ in $C_n$, say $s_i$ and $s_j$. Consider the set $D_S$ defined as follows:
\[ D_S = \bigcup_{i=1}^{c(C_n\setminus S)} \{v \in V(H_i) ~:~ \deg_{H_i}(v) = 2 \}. \]%\{ v\in V(G): \text{there exists }i \in [c(C_n\setminus S)]  \text{ such that } v\in V(H_i) ~\text{with}~ \deg_{H_i}(v)=2\}.\]
Note that each $H_i$ is a path graph. So, $D_S$ is the collection of all vertices of $H_i$ except the leaves. Let $C_1(S) := \{i \in [c(C_n\setminus S)] ~:~ |V(H_i)| =1\}$ and $C_2(S) = [c(C_n \setminus S)] \setminus C_1(S)$. % denote collection of all $H_i$'s, such that $|V(H_i)| =1$ and $C_2(S)$ denotes the collection of all $H_i$'s with $|V(H_i)| \geq 2$. 
By the construction of the set $D_S$, \[ |D_S| = \sum_{i\in C_1(S)} ( |V(H_i)| -1) + \sum _{i\in C_2(S)} (|V(H_i)|-2).\]
The first term in the above expression is zero, but we will keep that for computational purposes. Now we construct the following polynomial $f_S$ as

\[ f_S = \prod_{s\in S} f_{s_i s_j}\prod_{i\in D_S} x_i.\]  

\begin{lemma} \label{lemma_5}
    Let $C_n$ be a cycle and $f_S$ be the above polynomial. Then, $(J_{C_n}:f_S) = P_S$.
\end{lemma}

\begin{proof}
To prove that $(J_{C_n}:f_S) =P_S$, it is sufficient to show that $f_S\notin P_S$ and $P_Sf_S\subseteq J_{C_n}$. %First, we will show that $f_S \notin P_S$. 
Since $S$ is a cut set of $C_n$, for every $s\in S$, $s_i$ and $s_j$ are in distinct connected components of ${C_n}\setminus S$, where $N_{C_n}(s)= \{s_i,s_j\}$. Hence 
$f_{s_i s_j} \notin P_S \text{ for all } s\in S.$
The variables $x_i$ or $y_i$ belong to $P_S$ if and only if $i$ belongs to $S$. This implies that $x_j \notin P_S$ for all $j \in D_S$ as $D_S\subseteq V({C_n}\setminus S)$.
Since $P_S$ is a prime ideal, we get $f_S \notin P_S.$

It remains to show that  $P_Sf_S \subseteq J_{C_n}$. We show that for any generator $g$ of $P_S$, $gf_S \in J_{C_n}$. If $g$ equals to $x_s$ or $y_s$, then  \[x_sf_{s_is_j} = x_{s_i}f_{ss_j} - x_{s_j}f_{ss_i} \in J_{C_n}~ \text{ and } ~y_sf_{s_is_j} = y_{s_i}f_{ss_j} - y_{s_j}f_{ss_j} \in J_{C_n}.\] 
%Similarly, \[y_sf_{s_is_j} = y_{s_i}f_{ss_j} - y_{s_j}f_{ss_j} \in J_{C_n}.\]
Since $f_{s_is_j}$ is a divisor of $f_S$, 
$x_sf_S\in J_{C_n} ~\text{and}~ y_sf_S \in J_{C_n} ~ \text{for all $s\in S$}.$

Now, suppose $g = f_{lk} \in J_{\Tilde{H}_i}$ for some $i\in [c(C_{n}\setminus S)]$. If $f_{lk} \in J_{C_n}$, then $f_{lk}f_{S}\in J_{C_n}$. Otherwise, there must be a path from $l$ to $k$ in $H_i$. Let us assume that the path is $l=v_0,v_1,\ldots, v_t, v_{t+1} = k$. Then we know that $(\prod_{j=1}^{t}x_{v_j})f_{lk} \in J_{C_n}.$
Since each $H_i$ itself is a path, $\deg_{H_i}(v_j) = 2$ for every $j\in[t]$. So, by construction of $D_S$, $v_j \in D_S$ for all $j\in [t]$. This implies $\prod_{j=1}^{t}x_{v_j}$ is a divisor of $f_S$. Hence, $f_{lk}f_S\in J_{C_n}$. 
This completes the proof.
\end{proof}

Using this lemma, we obtain a combinatorial upper bound of local $\mathrm{v}$-number for cycle graphs. For $C_n$, it is easy to observe that $c(C_n\setminus S)=\vert S\vert$ for all $S\in \mathcal{C}(C_{n})$.

\begin{theorem} \label{thm_7}
Let $C_n$ be a cycle graph on $[n]$ vertices and $S\in \mathcal{C}(C_n)$. Then $\mathrm{v}_S(J_{C_n}) \leq n - |C_2(S)|.$ In particular, $\mathrm{v}(J_{C_n}) \leq n - |C_2(S)|$ for all $S\in\mathcal{C}(C_n)$.

\end{theorem}
%(n-2) - \lfloor \frac{n-2}{3}\rfloor
\begin{proof}
It is enough to prove the first inequality. From the definition of $\mathrm{v}$-number and  \Cref{lemma_5}, we have the following:
\begin{align*}
    \mathrm{v}_S(J_{C_n}) 
    & \leq \deg(f_S)\\
    & = 2|S| + |D_S|\\
    &= 2|S| + \sum_{i\in C_1(S)} ( |V(H_i)| -1) + \sum _{i\in C_2(S)} (|V(H_i)|-2) \\
    &= \{|S| + (\sum_{i\in C_1(S)} |V(H_i)|) + (\sum _{i\in C_2(S)} |V(H_i)|~) \}\\
    & \hspace{0.4cm} +|S|- |C_1(S)| - 2|C_2(S)|\\
    &= n + |S| - ( |C_1(S)| + |C_2(S)|) - |C_2(S)|\\
    &= n+ |S| - | c_{C_n}(S)| - |C_2(S)| \\
    &= n + |S| -|S|- |C_2(S)| \hspace{0.4cm} \text{ (since $c({C_n}\setminus S) = |S| $})\\
    & = n - |C_2(S)|.
\end{align*}
    This completes the proof.  
\end{proof}

%\textcolor{blue}{I think Lemma 4.10, 4.11 should be proposition}

\begin{proposition} \label{lemma_6}
  $\mathrm{v}_S(J_{C_4}) = 2$, for all $S\in \mathcal{C}(C_4)$.
\end{proposition}

\begin{proof}
It follows from \Cref{thm_3}, $(1)$, that $\vv(J_G) = \vv_{\emptyset}(J_G) = 2$. Let $S \in \mathcal{C}(C_4)$ such that $S \neq \emptyset$. Then $G\setminus S$ will be the disjoint union of two vertices of $C_4$. Let us denote them by $a_1$ and $a_2$. Then $(J_{C_4}:f_{a_1a_2}) = P_S(C_4)$. Since $\mathrm{v}(J_{C_4}) = 2$, we get $\mathrm{v}_S(J_{C_4}) = 2$.
\end{proof}

\begin{proposition}\label{lemma_7}
    $\mathrm{v}_S(J_{C_5}) = 3$, for all $S\in \mathcal{C}(C_5)$.
\end{proposition}

\begin{proof}
    In $C_5$, the cardinality of any minimal completion set is $3$, so, $\mathrm{v}_{\emptyset}(J_{C_5}) = 3$ \cite[Theorem 3.6]{ambhore2023v}. Since $C_5$ is not a cone graph and also does not belong to the class considered in Theorem \ref{thm_3}, we have $\mathrm{v}(J_{C_5}) \geq 3$. Hence, $\mathrm{v}(J_{C_5})= \mathrm{v}_{\emptyset}(J_{C_5}) = 3$. Let $S \in \mathcal{C}(C_5)$ be such that $S \neq \emptyset$. Then $C_5\setminus S$ is a disjoint union of an edge $\{a_1,a_2\}$ and a vertex $a_3$. It is straight forward to verify that $(J_{C_5}: f_{a_1a_3}x_{a_2}) = P_S(C_5)$. Since $\mathrm{v}(J_{C_5}) = 3$, we have $\mathrm{v}_S(J_{C_5}) = 3$.   
\end{proof}

\begin{corollary} \label{cor_cycle_v}
$\mathrm{v}(J_{C_n}) \leq \lceil \frac{2n}{3}\rceil$.
\end{corollary}

\begin{proof}
If $n= 1,2$ or $3$, then $C_n$ is a complete graph, and thus, $\mathrm{v}(J_{C_n}) = 0$. % as $J_{C_n}$ is a prime ideal for $n=1,2,3$. 
If $n=4,5$, then from  \Cref{lemma_6} \& \Cref{lemma_7}, it follows that 
\[\mathrm{v}(J_{C_n}) \leq \lceil \frac{2n}{3}\rceil.\]
Now suppose $ n \geq 6$. Then, take 
$S_0 = \left\{3,6,9,..., 3 \left\lfloor\frac{n}{3} \right\rfloor\right\}.$
Here, $|C_2(S)| = |C(G\setminus S)| = |S| = \lfloor\frac{n}{3}\rfloor$. By \Cref{thm_7}, $\mathrm{v}(J_{C_n}) \leq \mathrm{v}_{S_0}(J_{C_n}) \leq  n - \lfloor\frac{n}{3} \rfloor = \lceil \frac{2n}{3}\rceil.$
    \end{proof}

For $n \geq 6$, the choice of $S_0$ not only gives an upper bound of  $\mathrm{v}(J_{C_n})$, but also satisfies the property that $|C_2(S_0)| = \mathrm{max} \{ |C_2(S)| : S\in \mathcal{C}(C_n)\}$. From the above theorem, $n - |C_2(S)|$ is minimum if and only if $|C_2(S)|$ is maximum and this happens if and only if the number of connected components in $G\setminus S$ with at least two vertices is maximum. By our choice of $S_0$, only one connected component of $G\setminus S_0$ can have more than two vertices (if $3\nmid n $), but all other connected components have exactly two vertices. So, if one can show that for $n \geq 6$, $\mathrm{v}_S(J_{C_n}) = n - |C_2(S)|$, then $\mathrm{v}(J_{C_n})= \lceil \frac{2n}{3}\rceil$. On the other hand, Theorem \ref{thm_6} suggests the possible value of $\vv$-number of binomial edge ideals of a binary tree. Based on these results and Macaulay2 computations, we conjecture the following:

\begin{conjecture}\label{conjcycle}
Let $C_n$ denote the cycle graph on $n$ vertices and $B_n$ denote the binary tree of level $n$. Then
\begin{enumerate}
    \item $\mathrm{v}(J_{C_n}) =\lceil \frac{2n}{3}\rceil$ for all $n \geq 6$;

    \item $\mathrm{v}(J_{B_n}) = 2^{n-1} + \mathrm{v}(J_{B_{n-3}})$ for all $n \geq 3$.
\end{enumerate}

\end{conjecture}

\section{The pair ($\vv$-number, regularity)}\label{sec:pair}
%Existence of graphs with arbitrary $\vv$ number and relation with regularity

The aim of this section is to address the question: \textit{Given positive integers $k \leq m$, does there exist a connected graph $G$ with $\vv(I) = k$ and $\reg(I) = m+1$, where $I$ is an associated edge ideal of $G$?}  First, we answer the question affirmatively for the monomial edge ideals. Then, we prove the existence of such graphs for binomial edge ideals when $m \geq 2k$.
\medskip

%In this section, we will prove that for given any positive integers $k$ and $m$, such that $m>2k$ and $k \geq 2$, there exists a graph $G(k,r)$ with $\mathrm{v}(J_{G(k,r)}) = k$ and $\reg\left(\frac{R}{J_{G(k,r)}}\right) = m$. %We discuss the construction of $G(k,r)$.
To study the $\vv$-number and regularity of monomial edge ideals, we need some prerequisite notions. Let us start by describing those first.
\medskip

Let $G$ be a simple graph. A subset $C\subseteq V(G)$ is called a \textit{vertex cover} of $G$ if $C\cap e\neq \emptyset$ for all $e\in E(G)$. If a vertex cover is minimal with respect to inclusion, then we call it a \textit{minimal vertex cover}. Also, a subset $A\subseteq V(G)$ is said to be \textit{independent} if no edge of $G$ is contained in $A$, and $A$ is said to be a \textit{maximal independent set} if it is maximal with respect to inclusion. 
\medskip

\begin{definition}{\rm
Let $G$ be a graph with $V(G)=\{x_{1},\ldots,x_{n}\}$. Corresponding to $A\subseteq V(G)$, we associate a square-free monomial $X_{A}:=\prod_{x_{i}\in A} x_{i}$ in the polynomial ring $\mathcal{R}=K[x_{1},\ldots,x_{n}]$. The \textit{edge ideal} of $G$, denoted by $I(G)$, is a square-free monomial ideal of $R$ defined as follows:
$$I(G):=\big<X_{e}\mid e\in E(G)\big>.$$
The ideal $I(G)$ is a radical ideal of $\mathcal{R}$ and the primary decomposition of $I(G)$ is given by
$$I(G)=\bigcap_{C\,\,\text{is a minimal vertex cover of}\,\,G}\big<C\big>.$$
Thus, the ideal generated by a minimal vertex cover of $G$ is an associated prime of $G$, and vice versa.
}
\end{definition}

For an independent set $A$ of a graph $G$, consider the following set 
$$N_{G}(A):=\{x_{i}\in V(G) : \{x_{i}\}\cup A\,\, \text{contains\,\,an\,\,edge\,\,of}\,\, G\}.$$

%Let us denote by $\mathcal{A}_{G}$ the collection of those independent sets $A$ of $G$ such that $\mathcal{N}_{G}(A)$ is a minimal vertex cover of $G$. Then by \cite[Lemma 3.4]{vedge}, we have $I(G):X_{A}=\big<\mathcal{N}_{G}(A)\big>$ for all $A\in \mathcal{A}_{H}$ and due to \cite[Theorem 3.5]{vedge}, we get
 
\begin{definition}{\rm
Let $G$ be a simple graph. A set of edges $\{e_1,\ldots,e_k\}$ of $G$ is said to form an \textit{induced matching} in $G$ if $e_i\cap e_{j}=\emptyset$ for all $i\neq j$ and there is no edge contained in $e_1\cup\cdots\cup e_k$ other than $e_1,\ldots,e_k$. The \textit{induced matching number} of $G$, denoted by $\mathrm{im}(G)$, is the maximum cardinality of an induced matching in $G$.
}
\end{definition}

Let us consider a pair $(k,m)\in \mathbb{Z}_{>0}\times \mathbb{Z}_{>0}$ with $k\leq m$. We now construct a graph $H(k,m)$ as follows:% (see \Cref{figedge}):
\begin{enumerate}
    \item[$\bullet$] $V(H(k,m))=\{x, x_i, y_i, z_i, x_1, y_{1s}, z_{1s} : 2\leq i\leq k, 1\leq s\leq m-k+1\}.$ 
    \item[$\bullet$] $E(H(k,m))=\{\{x,x_{i}\},\{x_{i},y_{i}\},\{y_i,z_i\}, \{x_1,y_{1s}\}, \{y_{1s},z_{1s}\} :  2\leq i\leq k, 1\leq s\leq m-k+1\}$
\end{enumerate}

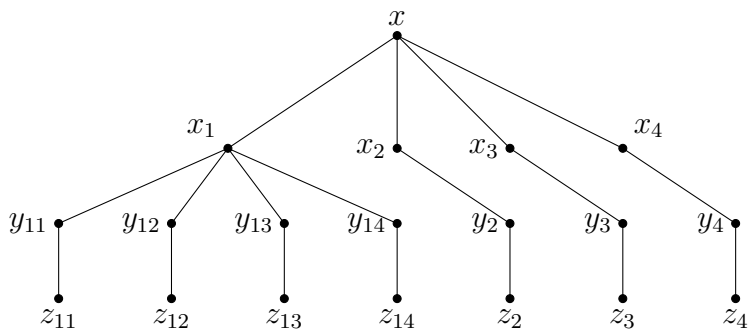
\begin{figure}[H]
    \centering
\begin{tikzpicture}[line cap=round,line join=round,>=triangle 45,x=1.5cm,y=1cm]

\draw (0,0)-- (0,1);
\draw (1,0)-- (1,1);
\draw (2,0)-- (2,1);
\draw (3,0)-- (3,1);
\draw (1.5,2)-- (0,1);
\draw (1.5,2)-- (1,1);
\draw (1.5,2)-- (2,1);
\draw (1.5,2)-- (3,1);
\draw (4,0)-- (4,1);
\draw (5,0)-- (5,1);
\draw (6,0)-- (6,1);
\draw (3,2)-- (4,1);
\draw (4,2)-- (5,1);
\draw (5,2)-- (6,1);
\draw (3,2)-- (3,3.5);
\draw (4,2)-- (3,3.5);
\draw (5,2)-- (3,3.5);
\draw (1.5,2)-- (3,3.5);

\draw (0,0) node[anchor=north] {$z_{11}$};
\draw (1,0) node[anchor=north] {$z_{12}$};
\draw (2,0) node[anchor=north] {$z_{13}$};
\draw (3,0) node[anchor=north] {$z_{14}$};
\draw (0,1) node[anchor=east] {$y_{11}$};
\draw (1,1) node[anchor=east] {$y_{12}$};
\draw (2,1) node[anchor=east] {$y_{13}$};
\draw (3,1) node[anchor=east] {$y_{14}$};
\draw (1.5,2) node[anchor=south east] {$x_1$};
\draw (4,1) node[anchor=east] {$y_2$};
\draw (5,1) node[anchor=east] {$y_3$};
\draw (6,1) node[anchor=east] {$y_4$};
\draw (4,0) node[anchor=north] {$z_{2}$};
\draw (5,0) node[anchor=north] {$z_3$};
\draw (6,0) node[anchor=north] {$z_4$};
\draw (3,2) node[anchor=east] {$x_2$};
\draw (4,2) node[anchor=east] {$x_3$};
\draw (5,2) node[anchor=south west] {$x_4$};
\draw (3,3.5) node[anchor=south] {$x$};

\begin{scriptsize}
\draw [fill=black] (0,0) circle (1.5pt);
\draw [fill=black] (1,0) circle (1.5pt);
\draw [fill=black] (2,0) circle (1.5pt);
\draw [fill=black] (3,0) circle (1.5pt);
\draw [fill=black] (0,1) circle (1.5pt);
\draw [fill=black] (1,1) circle (1.5pt);
\draw [fill=black] (2,1) circle (1.5pt);
\draw [fill=black] (3,1) circle (1.5pt);
\draw [fill=black] (1.5,2) circle (1.5pt);
\draw [fill=black] (4,1) circle (1.5pt);
\draw [fill=black] (5,1) circle (1.5pt);
\draw [fill=black] (6,1) circle (1.5pt);
\draw [fill=black] (4,0) circle (1.5pt);
\draw [fill=black] (5,0) circle (1.5pt);
\draw [fill=black] (6,0) circle (1.5pt);
\draw [fill=black] (3,2) circle (1.5pt);
\draw [fill=black] (4,2) circle (1.5pt);
\draw [fill=black] (5,2) circle (1.5pt);
\draw [fill=black] (3,3.5) circle (1.5pt);
\end{scriptsize}
\end{tikzpicture}
\caption{The graph $H(4,7)$}
    \label{figedge}
\end{figure}

For a simple graph $G$ on $\{x_1,\ldots,x_n\}$, let $I(G):= \langle\{x_ix_j ~:~ \{x_i,x_j\} \in E(G)\}\rangle$ be the edge ideal of $G$ in $\mathcal{R} = K[x_1,\ldots,x_n]$.
\begin{theorem}\label{thm_edge_pair}
    For any pair of positive integers $(k,m)$ with $k\leq m$, there exists a connected graph $G$ such that $$\left(\mathrm{v}(I(G)),\reg\left(\frac{\mathcal{R}}{I(G)}\right)\right)=\left(k,m\right).$$
\end{theorem}
\begin{proof}
For a given pair of positive integers $(k,m)$ with $k\leq m$, let $G=H(k,m)$. From the construction, it is clear that $G$ is connected. Due to \cite[Lemma 3.4]{v-edge}, there exists an independent set $A$ of $G$ such that $(I(G):X_{A})=(N_{G}(A))$, where $N_{G}(A)$ is a minimal vertex cover of $G$ and $\vert A\vert =\vv(I(G))$.  Since $N_{G}(A)$ is a minimal vertex cover of $G$, either $y_{1s}\in N_{G}(A)$ or $z_{1s}\in N_{G}(A)$ for each $s\in [m-k+1]$. If $y_{1s}\in N_{G}(A)$ for some $s\in [m-k+1]$, then either $x_1 \in A$ or $z_{1s}\in A$. If $y_{1s}\notin N_{G}(A)$ for all $s\in [m-k+1]$, then $z_{1s}\in N_{G}(A)$ for all $s\in [m-k+1]$, which imply $y_{1s}\in A$ for all $s\in [m-k+1]$. Now, for $2\leq i\leq k$, $\{y_i,z_i\}\in E(G)$ implies either $y_i\in N_{G}(A)$ or $z_{i}\in \mathcal{N}_{G}(A)$. If $y_i\in N_{G}(A)$, then either $x_{i}\in A$ or $z_{i}\in A$ and if $z_{i}\in \mathcal{N}_{G}(A)$, then $y_{i}\in A$. Considering all the possible cases, we get one of $x_i,y_i,z_i$ should belong to $A$ for each $2\leq i\leq k$ and one of $x_1,y_{1s},z_{1s}$ should belong to $A$ for each $s\in[m-k+1]$. Therefore, it is clear that $\vv(I(G))=\vert A\vert \geq k$. Now, if we take $A=\{x_1,\ldots,x_k\}$, then it is clear that $A$ is independent and $\mathcal{N}_{G}(A)$ is a minimal vertex cover of $G$. Hence, $\vv(I(G))=k$. From the structure of $G$, it is easy to observe that $\mathrm{im}(G)=m$. Since $G$ is a chordal graph, by \cite[Theorem 2.18]{zheng2004resolutions}, it follows that $\reg(\frac{\mathcal{R}}{I(G)})=m$. Thus, $(\mathrm{v}(I(G)),\reg(\frac{\mathcal{R}}{I(G)}))=(k,m)$.
\end{proof}

\begin{minipage}{\linewidth}
\begin{minipage}{0.5\linewidth}  
Write
\[G(k,r) = G_1 \cup_{v} \cdots \cup_{v} G_k \cup_v  H_1 \cup_v \cdots \cup_v H_r,\] where $G_i \cong K_{1,3}$ on the vertices $\{v,v_{i_1},v_{i_2}, v_{i_3}\}$ with $\deg_{G(k,r)}(v_{i_2}) = 3$, for all $i \in [k]$  and $H_j$ is the path on the vertices $\{v,u_{j_1},u_{j_2}\}$ with $\deg(u_{j_1}) = 2$.
\end{minipage}
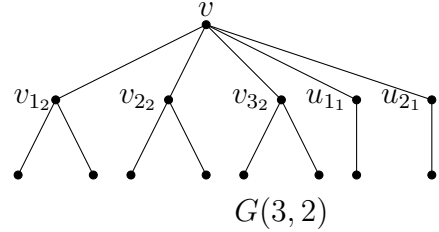
\begin{minipage}{0.45\linewidth}
\usetikzlibrary{arrows}
\begin{figure}[H]
\begin{tikzpicture}[scale=1]
\draw (4.,5.)-- (2.,4.);
\draw (2.,4.)-- (1.5,3.);
\draw (2.,4.)-- (2.5,3.);
\draw (4.,5.)-- (3.5,4);
\draw (3.5,4)-- (3.,3.);
\draw (3.5,4)-- (4.,3.);
\draw (4.,5.)-- (5.,4.);
\draw (5.,4.)-- (4.5,3.);
\draw (5.,4.)-- (5.5,3.);
\draw (4.,5.)-- (6.,4.);
\draw (6.,4.)-- (6.,3.);
\draw (4.,5.)-- (7.,4.);
\draw (7.,4.)-- (7.,3.);
\draw(5,2.5) node{$G(3,2)$};
\draw(4,5.2) node{$v$};
\draw(1.7,4) node{$v_{1_2}$};
\draw(3.1,4) node{$v_{2_2}$};
\draw(4.6,4) node{$v_{3_2}$};
\draw(5.6,4) node{$u_{1_1}$};
\draw(6.6,4) node{$u_{2_1}$};
\begin{scriptsize}
\draw [fill=black] (4.,5.) circle (1.5pt);
\draw [fill=black] (2.,4.) circle (1.5pt);
\draw [fill=black] (1.5,3.) circle (1.5pt);
\draw [fill=black] (2.5,3.) circle (1.5pt);
\draw [fill=black] (3.5,4) circle (1.5pt);
\draw [fill=black] (3.,3.) circle (1.5pt);
\draw [fill=black] (4.,3.) circle (1.5pt);
\draw [fill=black] (5,4) circle (1.5pt);
\draw [fill=black] (4.5,3.) circle (1.5pt);
\draw [fill=black] (6,3.) circle (1.5pt);
\draw [fill=black] (6.,4.) circle (1.5pt);
\draw [fill=black] (5.5,3.) circle (1.5pt);
\draw [fill=black] (7.,4.) circle (1.5pt);
\draw [fill=black] (7.,3.) circle (1.5pt);
\end{scriptsize}
\end{tikzpicture}
\end{figure}
\end{minipage}
\end{minipage}

%First we choose $k$ copies of $P_3$ and $r$ copies of $P_2$, $k>0$ and $r\geq 0$ and denote each $P_3$ by $G_i$, $i\in [k]$ and $P_2$ by $H_i$,  $i\in [r]$. We further assume that
%\[ V(G_i) = \{ v_{i_1}, v_{i_2}, v_{i_3}\} ~\text{for}~ i\in [k], \text{ where $deg_{G_i}(v_{i_2}) =2$}, \]
%and
%\[ V(H_j) = \{u_{j_1}, u_{j_2} \}  ~\text{ for}~ j\in [r].\]
%Now we take,
%\[ V(G(k,r)) = \{v\} \bigcup_{i\in [k]} V(G_i) \bigcup_{i\in[r]} V(H_i),\]
%where $v$ is an external vertex, and
%\[ E(G(k,r)) = \bigcup_{i\in[k]} E(G_i) \bigcup_{i\in[r]} E(H_i)  \bigcup \bigl\{ \{v, v_{i_2}\}, \{v,u_{j_1} \} : i\in [k], j\in [r] \bigr \} .\]
\noindent Observe that every cut set of $G(k,r)$ is a subset of $\{v, v_{i_2}, u_{j_1} : i\in [k], j\in [r] \}$. First, we understand the local $\vv$-number of $J_{G(n,r)}$ with respect to a cut set.

\begin{proposition}\label{lemma_3}
For every $S\in \mathcal{C}(G(k,r))$, $\mathrm{v}_S (J_{G(k,r)}) \geq k$.
\end{proposition}

\begin{proof}
Let $S$ be a cut set and $f$ be a homogeneous polynomial of $R$ with $(J_{G(k,r)}:f) = P_S$. Due to Proposition \ref{prop_1}, we may assume $\ini_{<} (f) \notin \ini_{<}(J_{G(k,r)})$. We show that $\deg(f) \geq k$.

Let $i \in [k]$ be such that $v_{i_2} \in S$. Then $f\in (J_{G(k,r)}: x_{v_{i_2}}) = J_{G(k,r)_{v_{i_2}}}$. Therefore, $\ini_{<} (f) \in \ini_{<}(J_{G(k,r)_{v_{i_2}}}).$
So, there exists a monomial $m$ in $\ini_{<}(J_{G(k,r)_{v_{i_2}}})$, corresponding to an admissible path $P$ in ${G(k,r)_{v_{i_2}}}$ such that $m$ divides $\ini_{<}(f)$. Since $\ini_{<} (f) \notin \ini_{<}(J_{G(k,r)})$, $P$ is not an admissible path of $G(k,r)$, i.e., the path $P$ must contain an edge from $E(G(k,r)_{v_{i_2}}) \setminus E(G(k,r)) = \bigl \{ \{v_{i_1}, v_{i_3}\}, \{ v, v_{i_1}\}, \{v, v_{i_3}\} \bigr \}$. Then either $v_{i_1}$ or $v_{i_3}$ is a vertex of $P$. Thus, at least one of $t_{v_{i_1}}$ or $t_{v_{i_3}}$ must divide $m$, where $t_l \in \{x_l,y_l\}$. %Hence, for each $i \in [k]$ such that $v_{i_2} \in S$, there is a variable of the ring $R$, corresponding to the indices $v_{i_1}$ or $v_{i_3}$ which divides $inf$.
Hence for $j \in \{1,3\}$,
\[ \prod_{\{i\in [k]~:~ v_{i_2}\in S\}} t_{v_{i_j}} | \ini_{<} (f).\]
%where $j \in \{1,3\}$.% and  for any vertex $l$ of $G(k,r)$ $t_l \in \{x_l,y_l\}$

Now, suppose $i \in [k]$ is such that $v_{i_2} \notin S$. Then, the graph $G_i\setminus \{v\}$ is an induced subgraph of a connected component of $G\setminus S$. Then $f_{v_{i_1}v_{i_3}} \in P_S \setminus J_{G(k,r)}$. Now,
$f  f_{v_{i_1}v_{i_3}}  \in J_{G(k,r)}$  implies that 
%\[ \implies inf \times in f_{v_{i_1}v_{i_3}} \in in J_{G(k,r)} \]
$\ini_{<} (f)  x_{v_{i_1}}y_{v_{i_3}} \in \ini_{<}(J_{G(k,r)})$ by assuming $v_{i_1}<v_{i_3}$. Suppose $m'$ is a monomial in $\ini_{<} (J_{G(k,r)})$, corresponding to an admissible path $P'$ in $G(k,r)$, that divides $\ini_{<} (f)  x_{v_{i_1}}y_{v_{i_3}}$. Since $\ini_{<} (f) \notin \ini_{<}(J_{G(k,r)})$, $P'$ must contain either $v_{i_1}$ or $v_{i_3}$. But any path of length at least two in $G(k,r)$, containing any of these vertices, must pass through $v_{i_2}$, i.e., $v_{i_2} \in V(P')$. Therefore, either $x_l$ or $y_l$ divides $\ini_{<} (f)$, where $l = v_{i_2}$. Hence,
       \[ \prod_{\{i \in [k] ~:~ v_{i_2}\notin S\}} t_{v_{i_2}} | \ini_{<} (f). \]
       
Combining the above two division statements, we get $k$ distinct variables that divide $\ini_{<} (f)$. Since $f$ is chosen arbitrarily, we have $\mathrm{v}_S(J_{G(k,r)}) \geq k$.
\end{proof}

\begin{theorem} \label{thm_4}
Let ${G(k,r)}$ be the graph as constructed above and $k \geq 2$. Then $\mathrm{v}(J_{G(k,r)}) = k$.
\end{theorem}

\begin{proof}
From \Cref{lemma_3}, we have $\mathrm{v}(J_{G(k,r)}) \geq k$. Therefore, to prove the assertion of the theorem, we need to show that $\mathrm{v}_T(J_{G(k,r)}) \leq k$ for some $T \in \mathcal{C}({G(k,r)})$. Let $T = \{v\}$ and     
    \[ f = \left(\prod_{i\in [k-2]} x_{v_{i_2}} \right)  f_{v_{(k-1)_2}v_{k_2}}\]

\noindent
\textbf{Claim:} $(J_{G(k,r)}:f) = P_T$. \vskip 2mm \noindent
\textit{Proof of the Claim:} Note that $v_{i_2} \notin T$ for all $i\in[k]$, and since $v_{(k-1)_2}$ and $v_{k_2}$ are in distinct connected components of ${G(k,r)}\setminus T,$ $f_{v_{(k-1)_2}v_{k_2}} \notin P_S$ and hence $f\notin P_T$. Therefore, $(J_{G(k,r)}:f) \subseteq P_T$. To prove the reverse inclusion, first observe that the generators of $P_T$ which are not in $J_{G(k,r)}$ are the variables $x_v,y_v$ and binomials $f_{v_{i_1}v_{i_3}}$, for $i \in [k]$. Since  $v_{(k-1)_2}$ and $v_{k_2}$ are adjacent to $v$ in $G(k,r)$, 
$t_v  f_{v_{(k-1)_2}v_{k_2}} \in J_{G(k,r)},$ where $t_v \in \{x_v,y_v\}$.
%The connected components of ${G(k,r)}\setminus T$ are $G_i \setminus\{v\}$ and $H_j\setminus\{v\}$, for $i\in[k]$ and $j\in [r]$. Thus, $f_{v_{i_1}v_{i_3}} \in P_T$ for all $i\in[k]$. 
Since $v_{i_1}$ and $v_{i_3}$ are neighbors of $v_{i_2}$, 
$x_{v_{i_2}}  f_{v_{i_1}v_{i_3}} \in J_{G(k,r)}.$ This implies that 
%for all $i \in [k-2]$,
% \[ \prod_{i\in [k-2]} x_{v_{i_2}} \times f_{v_{i_1}v_{i_3}} \in J_{G(k,r)} \] and  for $i= k-1, k$,
%
$f_{v_{i_1}v_{i_3}}f \in J_{G(k,r)}$ and $t_vf\in J_{G(k,r)}$, where $t_v \in \{x_v,y_v\}$.
This proves that $P_T \subseteq (J_{G(k,r)} : f)$. Hence, $ \mathrm{v}_T(J_{G(k,r)}) \leq \deg(f) = k.$ 
\end{proof}

\begin{corollary}
For any non-negative integer $k$, there exists a connected graph $G$ with $\mathrm{v}(J_G) = k$.
\end{corollary}  

\begin{proof}
If $k=0$, take $G$ to be any complete graph, as the binomial edge ideal of a complete graph is a prime ideal. If $k = 1$, then take $G$ to be any non-complete cone graph due to \cite[Theorem 3.20]{ambhore2023v}. If $k \geq 2$, then taking $G = G(k,r)$, we get $\vv(J_{G})=k$ by Theorem \ref{thm_4}, where $r$ is any non-negative integer.
 \end{proof}

Observe that the $\mathrm{v}$-number of the graph $G(k,r)$ does not depend on $r$. Now, our goal is to show that the regularity of $\frac{R}{J_{G(k,r)}}$ is $2k+r$. For that purpose, we need to study the regularity of complete graphs with whiskers. 

\begin{minipage}{1\linewidth}
\begin{minipage}{0.6\linewidth} 
For example, in \Cref{fig:com_whis}, take $G$ to be $K_4$ with $4$ whiskers attached to $3$ vertices of $K_4$. Then the number of cut vertices in $G$ is $3$, and the regularity of $\frac{R}{J_G} = 4$. We generalize this result for an arbitrary complete graph $K_m$ with whiskers, where $m>1$.
\end{minipage}
\begin{minipage}{0.4\linewidth} 

\begin{figure}[H]
\begin{tikzpicture}[line cap=round,line join=round,>=triangle 45,x=0.6cm,y=0.6cm]
%\clip(-0.08945612161226846,-0.11324165163341693) rectangle (7.249157211662567,6.9460806230259715);
\draw (2,2)-- (2,4);
\draw (2,2)-- (5,2);
\draw (5,2)-- (5,4);
\draw (2,4)-- (5,4);
\draw (2,2)-- (5,4);
\draw (2,4)-- (5,2);
\draw (2,2)-- (0,1);
\draw (2,4)-- (0,5);
\draw (5,4)-- (7,5);
\draw (5,4)-- (7,3);
%\draw (1,1.5) node[anchor=north west] {$G :$ $K_{4}$ with whiskers};
\begin{scriptsize}
\draw [fill=black] (2,2) circle (1.5pt);
\draw [fill=black] (2,4) circle (1.5pt);
\draw [fill=black] (5,2) circle (1.5pt);
\draw [fill=black] (5,4) circle (1.5pt);
\draw [fill=black] (0,1) circle (1.5pt);
\draw [fill=black] (0,5) circle (1.5pt);
\draw [fill=black] (7,5) circle (1.5pt);
\draw [fill=black] (7,3) circle (1.5pt);
\end{scriptsize}
\end{tikzpicture}
\caption{ $K_{4}$ with whiskers}
    \label{fig:com_whis}
\end{figure}
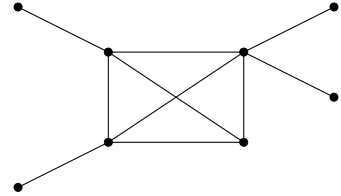
\end{minipage}
\end{minipage}

\begin{proposition}\label{lemma_4}
If $G$ is a complete graph with whiskers attached to $l$ distinct vertices, then $\reg (\frac{R}{J_G}) = l+1$.
\end{proposition}

\begin{proof}
Let $K_m (m>1)$ be the complete graph on $\{u_1, \ldots, u_m\}$ and $G$ be obtained by attaching whiskers on $\{u_1, \ldots, u_l\}$. Suppose $H$ is a subgraph of $G$, constructed by attaching a whisker each on $\{u_1, \ldots, u_l\}$. Then it follows from \cite[Theorem 3.1]{jayanthan2019regularity}) that
\[ \reg \left(\frac{R}{J_H}\right) = \reg \left(\frac{R}{J_{K_m}}\right) + \sum_{i\in[l]} \reg \left(\frac{R}{J_{K_2}}\right) = 1 + \sum_{i\in[l]}1 = l+1.\] 
Since $H$ is an induced subgraph of $G$, $\reg (\frac{R}{J_G}) \geq \reg (\frac{R}{J_H}) = l+1 $ \cite[Corollary 2.2]{matsuda2013regularity}. We now show that $\reg(\frac{R}{J_G}) \leq l+1$ by induction on $l$.
%Let us proceed by induction on $k$. 
If $l=0$, then $G$ is a complete graph, and thus, $\reg(S/J_G) = 1$ by \cite[Theorem 2.1]{kiani2012binomial}. If $m =2$, then for any $l > 0$,  $G$ is a caterpillar tree and hence $\reg(R/J_G) = l+1$, by Theorem 4.1, \cite{chaudhry2016binomial}. Now assume that $m > 2$ and $l >0$. Set $u = u_1$. Then we have the following short exact sequence \cite{ene2011cohen}.

\begin{eqnarray}\label{ohtani_ses}
0 & \rightarrow & \frac{R}{J_G} \rightarrow \frac{R}{J_{G_u}} \oplus \frac{R}{J_{G \setminus \{u\}}+(x_u,y_u)} \rightarrow \frac{R}{J_{G_u \setminus \{u\}}+(x_u,y_u)} \rightarrow 0.
\end{eqnarray}
 
\noindent Now $G_u$ and $G_u \setminus \{u\}$ both are complete graphs with whiskers on $l-1$  vertices, i.e., having $l-1$  cut vertices. The graph $G\setminus \{u\}$ also has $l-1$ cut vertices, and it is the union of a complete graph with whiskers and some isolated vertices. Since the isolated vertices do not contribute to $J_{G\setminus \{u\}}$, we can think of $G\setminus \{u\}$ as a connected graph. Thus, by induction hypothesis,
 \[ \reg\left(\frac{R}{J_{G_u}}\right) = \reg \left(\frac{R}{J_{G \setminus \{u\}}+(x_u,y_u)}\right) = \reg \left(\frac{R}{J_{G_u \setminus \{u\}}+(x_u,y_u)}\right) = l.\]
 Hence, $\reg (\frac{R}{J_G}) \leq \text{max} \{ \reg(\frac{R}{J_{G_u}}), \reg(\frac{R}{J_{G \setminus \{u\}} + (x_u,y_u)}), \reg (\frac{R}{J_{G_u \setminus \{u\}} + (x_u,y_u)})+1 \} = l+1$. 
\end{proof}

We now compute the regularity of $R/J_{G(k,r)}$.
\begin{theorem} \label{thm_5}
For every $k>0$ and $r\geq 0$, $\reg \left(\frac{R}{J_{G(k,r)}}\right) = 2k +r $ 
\end{theorem}

 \begin{proof}
Let $G=G(k,r)$ be as described early in this section and let $v \in V(G)$ be as given in the construction.
%the highest degree vertex of $V(G(k,r))$, as described in the structure of $G(k,r)$. 
Then $G \setminus \{v\}$ is the union of $k$ copies of $P_3$ and $r$ copies of $P_2$. %Since $P_3$ is gluing of two $K_2$ and $\reg(\frac{R}{J_{K_2}}) = 1$ (\cite{kiani2012binomial}). From (Theorem 3.1, \cite{jayanthan2019regularity}), $\reg(\frac{R}{J_{P_3}}) = 2$. 
Hence, $\reg(\frac{R}{J_{G\setminus \{v\}}}) = 2k + r$ which implies that $\reg(\frac{R}{J_{G}}) \geq 2k+r$ \cite[Corollary 2.2]{matsuda2013regularity}.
%     From \cite{ene2011cohen}, we have the following short exact sequence.

% \[ 0 \rightarrow \frac{R}{J_{G}} \rightarrow \frac{R}{J_{G_v}} \oplus \frac{R}{J_{G \setminus \{v\}}} \rightarrow \frac{R}{J_{G_v \setminus \{v\}}} \rightarrow 0.\]
Now, $G_v$ and $G_v \setminus \{v\}$ are complete graphs with whiskers, and they have $k+r$ cut vertices. By \Cref{lemma_4},  $\reg \left(\frac{R}{J_{G_v}}\right) = \reg \left(\frac{R}{J_{G_v \setminus \{v\}}}\right) = k+r +1$. Therefore, it follows from the short exact sequence (\ref{ohtani_ses}) (replacing $u$ by $v$) that
\begin{eqnarray*}
\reg \left(\frac{R}{J_G}\right) & \leq & \max \left\{\reg\left(\frac{R}{J_{{G}_v} + m_{\{v\}}}\right), \reg\left(\frac{R}{J_{{G} \setminus \{v\}} + m_{\{v\}}}\right), \reg \left(\frac{R}{J_{{G}_v \setminus \{v\}} + m_{\{v \}}}\right)+1 \right\} \\ & =  & 2k+r.    
\end{eqnarray*}
 Hence, $\reg \left(\frac{R}{J_{G(k,r)}}\right) = 2k +r $.
 \end{proof}
%\avj{I think we have been writing $m > 2k$ earlier. I think we only need $m\geq k$. Correct the earlier occurrences.}\textcolor{blue}{$m\geq 2k$ we need for binomial edge ideals}
\begin{corollary}\label{cor_binom_pair}
Given any pair of integers $(k,m)$ such that $k \geq 1$ and $m \geq 2k$, there exists a connected graph $G$ such that $(\mathrm{v}(J_G), \reg (\frac{R}{J_G})) = (k, m).$
 \end{corollary}
 \begin{proof}
     For $k=1$, the result follows from \cite[Theorem 4.11]{ambhore2023v}. For $k\geq 2$, take $G=G(k,r)$.
 \end{proof}

Note that $\reg (\frac{R}{J_G}))=1$ if and only if $G$ is a complete graph by \cite[Theorem 2.1]{kiani2012binomial}. On the other hand, if $G$ is a complete graph, then $J_{G}$ is a prime ideal, which gives $\vv(J_G)=0$. Thus, there exists no connected graph $G$ such that $(\mathrm{v}(J_G), \reg (\frac{R}{J_G})) = (1, 1)$. However, if we take $G=C_4$, then $(\mathrm{v}(J_G), \reg (\frac{R}{J_G})) = (2, 2)$. Therefore, due to Corollary \ref{cor_binom_pair} and \cite[Conjecture 5.3]{ambhore2023v}, we propose the following question:

\begin{question}
    Given a pair of positive integers $(k,m)$ with $2\leq k\leq m$, does there exist a connected graph $G$ such that $(\mathrm{v}(J_G), \reg (\frac{R}{J_G})) = (k, m)?$
\end{question}

\section{Proof of a conjecture on $\vv$-number for binomial edge ideals}\label{sec:conj}

%v-numbers of binomial edge ideals with linear powers/of graded ideals with linear powers
In this section, we settle the conjecture \cite[Conjecture 3.5]{bms24} of Biswas et al. for the class of binomial edge ideals.
The following proposition is due to \cite[Theorem 2.1]{kiani2012binomial} and \cite[Theorem 4.1]{ert22}, which classifies all the binomial edge ideals having linear powers.

\begin{proposition}\label{thm_lp}
    Let $G$ be a simple graph. Then the following are equivalent:
    \begin{enumerate}
        \item $J_G$ has a linear resolution;
        \item $J_G$ has linear powers;
        \item $G$ is complete.
    \end{enumerate}
\end{proposition}

\begin{proof}
    (1) $\implies$ (3) follows from \cite[Theorem 2.1]{kiani2012binomial} and (2) $\implies$ (1) is obvious.\par 
    (3) $\implies$ (2): Since $G$ is a complete graph, $G$ is closed. Thus, by \cite[Theorem 4.1]{ert22}, we have $\reg(\frac{R}{J_{G}^k})=2k-1$ for all $k\geq 1$. Hence, $J_{G}$ has linear powers.
\end{proof}

We now prove \cite[Conjecture 3.5]{bms24} for the class of binomial edge ideals.

\begin{theorem}\label{thm_vlp}
    Let $G$ be a simple graph. If $J_G$ has a linear resolution (equivalently, $J_G$ has linear powers), then $\vv(J_{G}^k)=\alpha(J_G)k-c(J_G)=2k-2$ for all $k\geq 1$. 
\end{theorem}

\begin{proof}
Due to Theorem \ref{thm_lp}, $G$ is a complete graph. Thus, $J_G$ is a prime ideal, which implies that $\vv(J_G)=0$. Since $J_G$ is a determinantal ideal,  $J_{G}^{(k)}=J_{G}^k$ for all $k\geq 1$, \cite{dep80}. This gives $\mathrm{Ass}(J_{G}^k)=\{J_G\}$ for all $k\geq 1$. We may assume that $V(G)=[n]$ with $n\geq 2$. Let us consider $f=f_{ij}$ for any $i,j\in [n]$ with $i\neq j$. Then $f\in J_G$, and we claim that $(J_{G}^k:f^{k-1})=J_{G}$. It is clear that $J_{G}\subseteq (J_{G}^{k}:f^{k-1})$. Since $\deg(f^{k-1})= 2k-2$ and $\alpha(J_{G}^{k})=2k$, $f^{k-1}\not \in J_{G}^{k}$. Thus, $(J_{G}^k:f^{k-1})$ is a proper ideal of $S$, which gives $\mathrm{Ass}(J_{G}^k:f^{k-1})\subseteq \mathrm{Ass}(J_{G}^{k})=\{J_G\}$. Hence, $(J_{G}^k:f^{k-1})=J_{G}$. Therefore, $\vv(J_{G}^{k})\leq 2k-2$. Thanks to \cite[Lemma 3.1]{bms24}, $\vv(J_{G}^{k})= 2k-2$ for all $k\geq 1$. 
\end{proof}

\bibliographystyle{plain}
\bibliography{sample}

\end{document}